\documentclass[12pt,a4paper,english,reqno]{amsart}
\usepackage{verbatim}
\usepackage{amsthm}
\usepackage{amscd}
\usepackage{amstext}
\usepackage{amssymb}
\usepackage[all]{xy}

\makeatletter
\@namedef{subjclassname@2010}{%
  \textup{2010} Mathematics Subject Classification}
\makeatother

\numberwithin{equation}{section}
\numberwithin{figure}{section}
\theoremstyle{plain}
\newtheorem{thm}{Theorem}[section]
\newtheorem{prp}[thm]{Proposition}
\newtheorem{cor}[thm]{Corollary}
\newtheorem{lem}[thm]{Lemma}
\newtheorem*{thm-nonum}{Theorem}

\newtheorem{claim}{Claim}

\theoremstyle{definition}
\newtheorem{dfn}[thm]{Definition}
\newtheorem{exm}[thm]{Example}

\theoremstyle{remark}
\newtheorem{rmk}[thm]{Remark}
\newtheorem{ntn}[thm]{Notation}

\usepackage{amscd, amssymb,}
\xyoption{2cell}
\UseTwocells

\newcommand{\xyR}[1]{%
\xydef@\xymatrixrowsep@{#1}}

\newcommand{\xyC}[1]{%
\xydef@\xymatrixcolsep@{#1}}

\def\al{\alpha}
\def\be{\beta}

\def\de{\delta}
\def\ep{\varepsilon}
\def\ze{\zeta}
\def\et{\eta}
\def\th{\theta}

\def\la{\lambda}
\def\ro{\rho}

\def\ta{\tau}

\def\ph{\phi}
\def\ps{\psi}

\def\De{\Delta}

\def\om{\omega}

\def\Aut{\operatorname{Aut}}

\def\calA{{\mathcal A}}
\def\calB{{\mathcal B}}
\def\calC{{\mathcal C}}

\def\calS{{\mathcal S}}

\def\op{^{\text{op}}}

\def\inv{^{-1}}

\def\implies{\text{$\Rightarrow$}\ }
\def\impliedby{\text{$\Leftarrow$}\ }

\def\incl{\hookrightarrow}
\def\iso{\cong}

\def\Ds{\bigoplus}

\def\dsm#1,#2..#3{\bigoplus_{{#1}={#2}}^{#3}}
\def\sm#1,#2..#3{\sum_{{#1}={#2}}^{#3}}

\def\id{1\kern-.25em{\text{{\rm l}}}} 

\def\isoto{\ \raise.8ex\hbox{$^{\sim}$}\kern-.7em\hbox{$\to$}\ }

\def\ya#1{\overset{#1}{\longrightarrow}}

\def\bg{%
\family{cmr}\size{20}{12pt}\selectfont}

\def\bigzerou{%
\smash{\lower1.7ex\hbox{\bg 0}}}

\def\repr[#1;#2;#3;#4;#5]{
\left(
\begin{matrix}#1\\#2\end{matrix}
#3
\begin{matrix}#4\\#5\end{matrix}
\right)}
\usepackage[all]{xy}
\numberwithin{equation}{section}
\numberwithin{figure}{section}

\def\C{\mathcal{C}}

\def\fstorbit{/\!_{_{1}}}
\def\sndorbit{/\!_{_{2}}}

\def\k{\Bbbk}
\def\kCat{\k\text{-}\mathbf{Cat}}
\def\To{\Rightarrow}
\def\Fgt{\operatorname{Fgt}}
\def\Inv{\operatorname{Inv}}
\def\Fun{\operatorname{Fun}}
\def\Cat{G\text{-}\mathbf{Cat}}
\def\GrCat{G\text{-}\mathbf{GrCat}}

\begin{document}

\title[2-categorical Cohen-Montgomery duality]
{A generalization of Gabriel's Galois covering functors II:
2-categorical Cohen-Montgomery duality}

\author{Hideto Asashiba}
\address{
Department of Mathematics,
Faculty of Science,
Shizuoka University,
836 Ohya, Suruga-ku,
Shizuoka, 422-8529, Japan
}
\email{asashiba.hideto@shizuoka.ac.jp}

\begin{abstract}
Given a group $G$,
we define suitable $2$-categorical structures
on the class of all small categories with $G$-actions
and on the class of all small $G$-graded categories,
and prove that 2-categorical extensions of the orbit category construction
and of the smash product construction
turn out to be 2-equivalences (2-quasi-inverses to each other),
which extends the Cohen-Montgomery duality.
Further we characterize equivalences in both 2-categories.
\end{abstract}


\thanks{This work is partially supported by Grant-in-Aid for Scientific Research
(C) 21540036 and Challenging Exploratory Research 25610003 from JSPS}

\maketitle

\section{Introduction}

Throughout this paper $G$ is a group and $\Bbbk$ is a commutative
ring, and all categories, functors, and algebras considered here are
assumed to be $\Bbbk$-linear unless otherwise stated.
This is a continuation of the paper \cite{Asa-cover}
and will be applied in subsequent papers \cite{Asa-Ki} and \cite{Asa}.

In \cite{Co-Mo} Cohen and Montgomery proved the following
(called the {\em Cohen-Montgomery duality}).

\begin{thm}
Let $G$ be a finite group of order $n$, $A$ an algebra with
a $G$-action, and $B$ a $G$-graded algebra.
Then we have isomorphisms
$$
\begin{aligned}
(A*G)\#G &\iso M_n(A)\\
(B\#G)*G &\iso M_n(B).
\end{aligned}
$$
\end{thm}
In the above, $*$ and $\#$ stand for
the skew group algebra construction and
the smash product construction, respectively and
$M_n(A)$ denotes the algebra of all
$n\times n$ matrices over $A$.
We can regard each algebra $A$ as a category with a single
object, and then $M_n(A)$ can be regarded as a category
with precisely $n$ objects that are isomorphic to each other,
and $A$ and $M_n(A)$ are equivalent as categories.

Already some attempts have been made to extend this theorem
so that it satisfies the following requirements.
\begin{enumerate}
\item[(a)]
Deal with an arbitrary group $G$;
\item[(b)]
Replace algebras by categories.
\end{enumerate}
For instance (a) 
was investigated in \cite{Be-dual}, \cite{Qu}, and
(b) was examined in \cite{C-M}, \cite{Asa-cover}.
To be more precise let $\calC$ be a category
with a $G$-action and $\calB$ a $G$-graded category.
Then
a $G$-graded category $\mathcal{C}/G$, called the \emph{orbit category}
of $\mathcal{C}$ by $G$ is constructed in
\cite{Asa-cover,C-M,Ke}
(this turns out to be also a generalization of the skew group algebra construction);
and a category $\mathcal{B}\#G$
with a free $G$-action, called the \emph{smash product}
of $\mathcal{C}$ and $G$ is constructed in \cite{C-M}; 
and in \cite{Asa-cover}
we defined a (weakly) $G$-equivariant equivalence
$\ep_{\calC} \colon \mathcal{C}\to(\mathcal{C}/G)\#G$
and a degree-preserving equivalence
$\om_{\calB} \colon \mathcal{B}\to(\mathcal{B}\#G)/G$.
This seems to give a full categorical generalization of Cohen-Montgomery duality.

Here recall the definition of equivalences between categories:
Categories (= objects) $\calA$ and $\calB$ are said to be {\em equivalent} if
there exist a pair of functors (= 1-morphisms) $E \colon \calA \to \calB$
and $F \colon \calB \to \calA$ in mutually reverse directions
such that there exist a pair of natural isomorphisms (= 2-isomorphisms)
$\ep \colon EF \To \id_{\calB}$ and $\et \colon \id_{\calA} \To FE$.
Namely, to define equivalences between objects in a categorical sense
we need
a 2-categorical structure in the class of objects.
In our case, the class $\Cat$ of all small $\k$-categories with $G$-actions and
the class $\GrCat$ of all small $G$-graded $\k$-categories
should have 2-categorical structures.
To insist that the above gives a full categorical generalization of Cohen-Montgomery duality
we have to have an affirmative answer to the following question:
\begin{enumerate}
\item[(i)] Are the $G$-equivariant equivalence $\ep_{\calC}$ and
the degree-preserving equivalence $\om_{\calB}$
obtained in \cite{Asa-cover}
equivalences defined by 2-categorical structures on $\Cat$ and $\GrCat$,
respectively?
\end{enumerate}
Once we have 2-categorical structures on  $\Cat$ and $\GrCat$,
it also becomes  important to consider the following question:
\begin{itemize}
\item[(ii)] Are $\ep_{\calC}$ and $\om_{\calB}$
2-natural in $\calC$ and in $\calB$?
\end{itemize}
These suggest us the following problem:
\begin{enumerate}
\item[(c)]
Not only give an equivalence for each individual category,
but extend it to a 2-equivalence between 2-categories of
$\k$-categories with $G$-action and of $G$-graded $\k$-categories.
\end{enumerate}
In this paper we will give a positive solution to the problem (c)
which includes affirmative answers to both (i) and (ii).
We also give characterizations of equivalences in 2-categories $\Cat$ and $\GrCat$
in terms of a half of a pair of functors in mutually reverse directions,
which give relationships between $G$-equivariant equivalences and equivalences
in $\Cat$
and between degree-preserving equivalences and equivalences in $\GrCat$.
The solution proceeds in the following steps: 
\begin{itemize}
\item to {\em suitably} define a 2-category $\Cat$ of all small $\k$-categories
with $G$-actions (Definition \ref{dfn:G-Cat})
and a 2-category $\GrCat$ of all small $G$-graded $\k$-categories
(Definition \ref{dfn:G-GrCat}); 
\item to extend the orbit category construction to a
2-functor
$$?/G\colon\Cat\to\GrCat$$
(Definition \ref{dfn:2-fun-?/G}) (this is given by the 2-universality of  the canonical functor $(P, \ps)$ that is a generalization
of Gabriel's Galois covering functor) 
and the smash product construction
to a 2-functor
$$?\#G\colon\GrCat\to\Cat$$
(Definition \ref{2-smash}); and
\item to prove the following (see Theorem \ref{main-thm} for detail):
\end{itemize}
\begin{thm}\label{thm-2-eq}
$?/G$ is strictly left $2$-adjoint to $?\# G$ and
they are mutual $2$-quasi-inverses $($in a weak sense$)$.
\end{thm}
Therefore in other words we obtain the following.

\begin{thm}\label{thm-2-mor}
Let $\calC, \calC' \in \Cat$ and $\calB, \calB' \in \GrCat$.
Then
\begin{enumerate}
\item there exists an equivalence $\calC \simeq (\calC/G)\# G$
$($in fact this is given by $\ep_{\calC}$ above$)$
in the $2$-category $\Cat$ that is $2$-natural in $\calC$;
\item there exists an equivalence $\calB \simeq (\calB\# G)/G$
$($in fact this is given by $\om_{\calB}$ above$)$
in the $2$-category $\GrCat$ that is $2$-natural in $\calB$;
\item there exists an {\em isomorphism}
$$
\GrCat(\calC/G, \calB) \iso \Cat(\calC, \calB \# G)
$$
of $\k$-categories that is $2$-natural in $\calC$ and $\calB$;
\item there exists an equivalence
$$
\Cat(\calC, \calC') \simeq \GrCat(\calC/G, \calC'/G)
$$
of $\k$-categories that is $2$-natural in $\calC$ and $\calC'$; and
\item there exists an equivalence
$$
\GrCat(\calB, \calB') \simeq \Cat(\calB \# G, \calB' \# G)
$$
of $\k$-categories that is $2$-natural in $\calB$ and $\calB'$.
\end{enumerate}
\end{thm}
Note that the statements (1) and (2)  above give affirmative answers to
both questions (i) and (ii).
We remark that the definition of {\em degree-preserving functors}
(= 1-morphisms in $\GrCat$) given here is slightly weakened than
that used in \cite{Asa-cover},
where degree-preserving functors were defined as strictly degree-preserving
functors in the sense of this paper
(see Definition \ref{dfn-graded-cat} (2), (3)).
This would be the most important point to establish our 2-equivalences
(see Remark \ref{rmk:necessity-weak} for the necessity of the weaker definition).
The results of this paper are applied at least in papers
\cite{Asa}, \cite{Asa-Ki} and \cite{Vo-Zv} so far.

For general 2-categorical notions we refer the reader to \cite{Bor} or \cite{Ke-St}.
In this paper 2-categories are strict 2-categories, and we use the word 
{}``strictly 2-natural transformation'' to mean
the 2-natural transformation in
a usual sense (e.g., as in \cite{Bor,Ke-St}),
and the word {}``2-natural transformation'' in a weak sense,
i.e., we only require that the
equalities defining the notion of usual 2-natural transformations
hold up to natural isomorphisms. Thus we use
the word {}``2-quasi-inverse''
in a weak sense (although in fact a half of the equalities to define
this notion hold strictly).

The paper is organized as follows.
In sections 2 and 3 we define the 2-category
$\Cat$ and the 2-category $\GrCat$, respectively.
In sections 4, 5 and 6 we recall from \cite{Asa-cover}
fundamental facts about $G$-covering,
the definition and characterizations of orbit categories, and
fundamental facts about smash products, respectively.
In section 7 we extend the orbit category construction
and the smash product construction to 2-functors $?/G$ and $?\#G$,
respectively, and give the precise
statement of the main result.
We also give a characterization of $G$-covering functors that
{\em induce degree-preserving functors} (Definition \ref{dfn:induce-degree-preserving}).
Section 8 is devoted to the proof of the main theorem.
Finally, in section 9 we characterize equivalences in the 2-categories $\Cat$ and $\GrCat$.

For categories $\calA$ and $\calB$ we write $\calA \iso \calB$ (resp.\ $\calA \simeq \calB$)
to express that they are isomorphic (resp.\ equivalent); and
the class of objects (resp.\ morphisms) in $\calA$ is denoted by
$\calA_0$ (resp.\ $\calA_1$).
We sometimes write ``$x \in \calA$'' as an abbreviation of ``$x \in \calA_0$''.
Natural transformations (and 2-morphisms in 2-categories) are expressed
by a double arrow symbol $\To$.

\section*{Acknowledgments}

I would like to thank Bernhard Keller and Dai Tamaki for
useful conversations.
The result was announced at the workshop of LMS Midlands Regional Meeting
2009 held at University of Leicester. 
Section 9 was added during my stay in Bielefeld in September, 2011.
I would like to thank C.\ M.\ Ringel, H.\ Krause and all members of the group
of representation theory of algebras for nice conversations and hospitality.
Finally, I would like to thank the referee for his/her useful comments
to make the paper more readable.
 
\section{The 2-category $G$-$\mathbf{Cat}$}

First in this section we define the 2-category of $G$-categories.

\subsection{$G$-categories}
\begin{dfn}
\label{dfn:G-action}
A $\k$-category with a $G$-\emph{action}, or a $G$-\emph{category} for short,
is a pair $(\mathcal{C},A)$ of a category $\mathcal{C}$
and a group homomorphism $A\colon G\to\Aut(\mathcal{C})$.
We set $A_{a}:=A(a)$ for all $a\in G$. If there is no confusion we
always denote $G$-actions by the same letter $A$, and simply write
$\mathcal{C}=(\mathcal{C},A)$.
\end{dfn}

\begin{ntn}
We denote by $\kCat$ the $2$-category of small $\k$-categories, $\k$-functors between them,
and natural transformations between $\k$-functors.
\end{ntn}

\begin{exm}\label{De-obj}
Any $\k$-category $\calB$ defines a $G$-category $\De\calB:= (\calB, A)$, where
$A \colon G \to \Aut(\calB)$ is the trivial $G$-action, namely it is defined by
$A_a:=\id_{\calB}$ for all $a \in G$.
We sometimes identify $\De\calB$ with $\calB$.
\end{exm}

\subsection{$G$-equivariant functors}
\begin{dfn}[{\cite[Definition 4.8]{Asa-cover}}]
Let $\mathcal{C}$ and $\mathcal{C}'$ be $G$-categories.
Then a $G$-\emph{equivariant} functor from $\calC$ to $\calC'$ is a pair
$(E,\rho)$ of a $\k$-functor
$E\colon\mathcal{C}\to\mathcal{C}'$ and
a family $\rho=(\rho_{a})_{a\in G}$ of natural isomorphisms
$\rho_{a}\colon A_{a}E\To EA_{a}$ ($a\in G$)
 such that the diagrams
 \[
\xymatrix{
A_{ba}E=A_{b}A_{a}E\ar@{=>}[r]^{A_{b}\rho_{a}}\ar@{=>}[rd]_{\rho_{ba}} &
A_{b}EA_{a}\ar@{=>}[d]^{\rho_{b}A_{a}}\\
 & EA_{ba}=EA_{b}A_{a}}
\]
commute for all $a,b\in G$.

A $\k$-functor $E \colon \calC \to \calC'$ is called
a \emph{strictly} $G$-\emph{equivariant} functor
if  $(E, (\id_E)_{a\in G})$ is a $G$-equivariant functor,
i.e., if $A_{a}E=EA_{a}$ for all $a\in G$.
\end{dfn}

\begin{rmk}
In the above
since $A_1 = \id$, we have $\ro_1x =\ro_1x \cdot \ro_1x$, and hence
$\ro_1x = \id_{Ex}$ for all $x \in \calC$.
Hence the natural requirement $\ro_1 = \id_E$ follows automatically from
the defining condition.
\end{rmk}

\begin{exm}\label{De-mor}
Any $\k$-functor $F\colon\mathcal{B}\to\mathcal{B}'$ defines a
strictly $G$-equivariant functor $\De F:= (F, (\id_F)_{a\in G}) \colon \De\calB \to \De\calB'$.
\end{exm}

\subsection{Morphisms of $G$-equivariant functors}
\begin{dfn}
Let $(E,\rho),(E',\rho')\colon\mathcal{C}\to\mathcal{C}'$ be $G$-\emph{equivariant}
functors. Then a \emph{morphism} from $(E,\rho)$ to $(E',\rho')$
is a natural transformation $\eta\colon E\To E'$ such that
the diagrams
$$
\xymatrix{
A_{a}E & EA_a\\
A_aE' &E'A_a
\ar@{=>}^{\ro_a}"1,1";"1,2"
\ar@{=>}_{\ro'_a}"2,1";"2,2"
\ar@{=>}_{A_a\et}"1,1";"2,1"
\ar@{=>}^{\et A_a}"1,2";"2,2"
}
$$
commute for all $a\in G$.
\end{dfn}
We define a composition of $G$-equivariant functors.
\begin{lem}
\label{lem:eqv-eqv-is-eqv}Let $ $$\xymatrix{\mathcal{C}\ar[r]^{(E,\rho)} & \mathcal{C}'\ar[r]^{(E',\rho')} & \calC''}
$ $ $be $G$-equivariant functors of $G$-categories. Then

$(1)$ $(E'E,((E'\rho_{a})(\rho_{a}'E))_{a\in G})\colon\mathcal{C}\to\mathcal{C}''$
is a $G$-equivariant functor, which we define to be the composite
$(E',\rho')(E,\rho)$ of $(E,\rho)$ and $(E',\rho')$.

$(2)$ If further $(E'',\rho'')\colon\mathcal{C}''\to\mathcal{C}'''$
is a $G$-equivariant functor, then we have \[
((E,\rho)(E',\rho'))(E'',\rho'')=(E,\rho)((E',\rho')(E'',\rho'')).\]
\end{lem}
\begin{proof}
Straightforward.
\end{proof}

\subsection{2-category $\Cat$}
\begin{dfn}\label{dfn:G-Cat}
A 2-category $G$-$\mathbf{Cat}$ is defined as follows.
\begin{itemize}
\item The objects are the small $G$-categories.
\item The 1-morphisms are the $G$-equivariant functors between objects.
\item The identity 1-morphism of an object $\calC$ is the 1-morphism
$(\id_\calC, (\id_{\id_\calC})_{a \in G})$.
\item The 2-morphisms are the morphisms of $G$-equivariant functors.
\item The identity 2-morphism of a 1-morphism $(E, \ro) \colon \calC \to \calC'$
is the identity natural transformation  $\id_E$ of $E$, which is clearly a 2-morphism.
\item The composition of 1-morphisms is the one given in the previous lemma.
\item The vertical and the horizontal compositions of 2-morphisms are given
by the usual ones of natural transformations.
\end{itemize}
\end{dfn}
\begin{prp}
The data above determine a $2$-category.\end{prp}
\begin{proof}
Straightforward.
\end{proof}

\begin{dfn}
Let $F$ and $F'$ be functors $\calB \to \calB'$ in $\kCat$, and
$\al \colon F \to F'$ a natural transformation.
Then we define a morphism $\De\ep \colon \De F \to \De F'$
of $G$-equivariant functors by setting $\De\ep:= \ep$.
This and the constructions given in Examples \ref{De-obj} and \ref{De-mor}
define a 2-functor $\De \colon \kCat \to \Cat$.
\end{dfn}

\section{The 2-category $G$-$\mathbf{GrCat}$}

In this section we cite necessary definitions and statements from
\cite[\S 5]{Asa-cover} and add new concepts and statements to define
the 2-category of $G$-graded categories. Here we modified the definition
of degree-preserving functors in order to include the functor $H$
(and hence the functors $\om'_\calB$ for all $G$-graded categories $\calB$, see
Definition \ref{dfn:om'})
in Proposition \ref{prp:cov-fun-wrt-smash} below because $H$ is
not degree-preserving in the sense of \cite{Asa-cover} in general
(see \cite[Remark 5.9]{Asa-cover} and Remark \ref{rmk:necessity-weak}).

\begin{dfn}\label{dfn-graded-cat}
$(1)$ A $G$-\emph{graded} $\k$-category is a category $\mathcal{B}$
together with a family of direct sum decompositions $\mathcal{B}(x,y)=\bigoplus_{a\in G}\mathcal{B}^{a}(x,y)$
$(x,y\in\mathcal{B})$ of $\Bbbk$-modules such that $\mathcal{B}^{b}(y,z)\cdot\mathcal{B}^{a}(x,y)\subseteq\mathcal{B}^{ba}(x,z)$
for all $x,y\in\mathcal{B}$ and $a,b\in G$. If $f\in\mathcal{B}^{a}(x,y)$
for some $a\in G$, then we set $\deg f:=a$.

$(2)$ A \emph{degree-preserving} functor is a pair $(H,r)$ of a
$\k$-functor $H\colon\mathcal{B}\to\mathcal{A}$ of $G$-graded categories
and a map $r\colon \mathcal{B}_0\to G$ such that
$$
H(\mathcal{B}^{r_{y}a}(x,y))\subseteq\mathcal{A}{}^{ar_{x}}(Hx,Hy)
$$
(or equivalently $H(\mathcal{B}^{a}(x,y))\subseteq\mathcal{A}{}^{r_{y}\inv ar_{x}}(Hx,Hy)$)
for all $x,y\in\mathcal{B}$ and $a\in G$.
This $r$ is called a \emph{degree adjuster} of $H$.

$(3)$ A $\k$-functor $H\colon\mathcal{B}\to\mathcal{A}$ of $G$-graded
categories is called a \emph{strictly} degree-preserving functor if
$(H,1)$ is a degree-preserving functor, where 1 denotes the constant
map $\mathcal{B}_0\to G$ with value $1\in G$, i.e., if $H(\mathcal{B}^{a}(x,y))\subseteq\mathcal{A}^{a}(Hx,Hy)$
for all $x,y\in\mathcal{B}$ and $a\in G$.

$(4)$ Let $(H,r),(I,s)\colon\mathcal{B}\to\mathcal{A}$ be degree-preserving
functors. Then a natural transformation $\theta\colon H\To I$ is
called a \emph{morphism} of degree-preserving functors if $\theta x\in\mathcal{A}{}^{s_{x}^{-1}r_{x}}(Hx,Ix)$
for all $x\in\mathcal{B}$.
\end{dfn}

The composite of degree-preserving functors can be made into again
a degree-preserving functor as follows.
\begin{lem}
Let ${}$$\xymatrix{\mathcal{B}\ar[r]^{(H,r)} & \mathcal{B}'\ar[r]^{(H',r')} & \mathcal{B}''}
$ be degree-preserving functors. Then \[
(H'H,(r_{x}r'_{Hx})_{x\in\mathcal{B}})\colon\mathcal{B}\to\mathcal{B}''\]
 is also a degree-preserving functor, which we define to be the \emph{composite}
$(H',r')(H,r)$ of $(H,r)$ and $(H',r')$.\end{lem}
\begin{proof}
Straightforward.
\end{proof}

\begin{dfn}\label{dfn:G-GrCat}
A 2-category $G$-$\mathbf{GrCat}$ is defined as follows.
\begin{itemize}
\item The objects are the small $G$-graded categories.
\item The 1-morphisms are the degree-preserving functors between objects.
\item
The identity 1-morphism of an object $\calB$ is the 1-morphism $(\id_\calB, 1)$.
\item The 2-morphisms are the morphisms of degree-preserving functors.
\item
The identity 2-morphism of a 1-morphism $(H, r)\colon \calB \to \calA$
is the identity natural transformation $\id_H$ of $H$, which is a 2-morphism
(because $(\id_{H})x =\id_{Hx} \in \calA^1(Hx, Hx) = \calA^{r_x\inv r_x}(Hx, Hx)$
 for all $x \in \calB$).
\item The composition of 1-morphisms is the one given in the previous lemma.
\item The vertical and the horizontal compositions of 2-morphisms are given
by the usual ones of natural transformations.
\end{itemize}
\end{dfn}
\begin{prp}
The data above determine a $2$-category.\end{prp}
\begin{proof}
Straightforward.
\end{proof}

\section{Covering functors}

Throughout sections 4 and 5, $\mathcal{C}$ is a $G$-category and $\mathcal{B}$
is a $\k$-category. In this section we cite definitions and statements
without proofs from \cite[\S 1]{Asa-cover}.

\subsection{$G$-invariant functors}
\begin{dfn}[{\cite[Definition 1.1]{Asa-cover}}]
A $G$-\emph{invariant} functor from $\calC$ to $\calB$
is a $G$-equivariant functor
$$
(F, \ph) \colon \calC \to \De\calB.
$$
We sometimes write this as $(F, \ph) \colon \calC \to \calB$.
\end{dfn}

\begin{rmk}
In the above the defining condition on $\ph = (\ph_a)_{a\in G}$ becomes as follows:
The diagrams
 \[
\xymatrix{F\ar@{=>}[r]^{\phi_{a}}\ar@{=>}[rd]_{\phi_{ba}} & FA_{a}\ar@{=>}[d]^{\phi_{b}A_{a}}\\
 & FA_{ba}=FA_{b}A_{a}}
\]
commute for all $a,b\in G$.
In particular, this implies
$\phi_{a}^{-1}=\phi_{a^{-1}}A_{a}$ for all
$a\in G$.
\end{rmk}

\subsection{Morphisms of $G$-invariant functors}
\begin{dfn}
\label{dfn:mor-of-inv}Let $(F,\phi),(F',\phi')$ be $G$-invariant
functors $\mathcal{C}\to\mathcal{B}$.
Then a \emph{morphism} of $G$-invariant functors
from $(F, \ph)$ to $(F', \ph')$ is just a morphism $\et$ of $G$-equivariant functors,
namely $\et$ is a natural transformation $F\to F'$ such that
the diagrams
$$
\xymatrix{
F & FA_a\\
F' &F'A_a
\ar@{=>}^{\ph_a}"1,1";"1,2"
\ar@{=>}_{\ph'_a}"2,1";"2,2"
\ar@{=>}_{\et}"1,1";"2,1"
\ar@{=>}^{\et A_a}"1,2";"2,2"
}
$$
commute for all $a\in G$.
\end{dfn}

\begin{ntn}
All $G$-invariant functors $\mathcal{C}\to\mathcal{B}$ and all morphisms
between them form a category, which we denote by $\Inv(\mathcal{C},\mathcal{B})$.
When both $\calC$ and $\calB$ are small categories, we have
$\Inv(\calC, \calB) = \Cat(\calC, \De\calB)$.
\end{ntn}
As a special case of Lemma \ref{lem:eqv-eqv-is-eqv}, the composite
of a $G$-invariant functor and a functor is made into again a $G$-invariant
functor:
\begin{lem}[{\cite[Lemma 1.4]{Asa-cover}}]
\label{lem:compn-of-inv-and-fun}Let $(F,\phi)\colon\mathcal{C}\to\mathcal{B}$
be a $G$-invariant functor and $H\colon\mathcal{B}\to\mathcal{A}$
a functor. Then $(HF,H\phi)\colon\mathcal{C}\to\mathcal{A}$ is again
a $G$-invariant functor, where $H\phi:=(H\phi_{a})_{a\in G}$. We
set $H(F,\phi):=(HF,H\phi)$.
\end{lem}

\subsection{$G$-covering functors}
\begin{ntn}
Let $(F,\ph)\colon\mathcal{C}\to\mathcal{B}$ be a $G$-invariant
functor and $x,y\in\mathcal{C}$.
Then we define homomorphisms $F_{x,y}^{(1)}:= (F, \ph)_{x,y}^{(1)}$
and $F_{x,y}^{(2)}:= (F, \ph)_{x,y}^{(2)}$ of $\Bbbk$-modules as follows.
\end{ntn}
\begin{eqnarray*}
F_{x,y}^{(1)}\colon\bigoplus_{a\in G}\mathcal{C}(A_{a}x,y) & \to & \mathcal{B}(Fx,Fy),\ (f_{a})_{a\in G}\mapsto\sum_{a\in G}F(f_{a})\cdot\phi_{a}x\\
F_{x,y}^{(2)}\colon\bigoplus_{b\in G}\mathcal{C}(x,A_{b}y) & \to & \mathcal{B}(Fx,Fy),\ (f_{b})_{b\in G}\mapsto\sum_{b\in G}\phi_{b^{-1}}(A_{b}y)\cdot F(f_{b})\end{eqnarray*}

\begin{prp}[{\cite[Proposition 1.6]{Asa-cover}}]
\label{prp:iso-iso}In the above, $F_{x,y}^{(1)}$ is an isomorphism
if and only if $F_{x,y}^{(2)}$ is.
\end{prp}

\begin{dfn}[{\cite[Definition 1.7]{Asa-cover}}]
\label{dfn:G-covering}Let $(F,\phi)\colon\mathcal{C}\to\mathcal{B}$
be a $G$-invariant functor. Then

(1) $(F,\phi)$ is called a $G$-\emph{precovering} if for each $x,y\in\mathcal{C}$,
$F_{x,y}^{(1)}$ is an isomorphisms (the latter is equivalent to saying
that $F_{x,y}^{(2)}$ is an isomorphism by Proposition \ref{prp:iso-iso});

(2) $(F,\phi)$ is called a $G$-\emph{covering} if it is a $G$-precovering
and $F$ is \emph{dense} (i.e., for each $y\in\mathcal{B}$ there
is an $x\in\mathcal{C}$ such that $Fx\cong y$ in $\mathcal{B}$).

\end{dfn}

\section{Orbit categories}

In this section we cite necessary definitions and statements without
proofs from \cite[\S 2]{Asa-cover} except for \S\ref{sub:eqv-inv-is-eqv}.
The symbol $\de_{a,b}$ stands for the Kronecker's delta below.

\subsection{Canonical $G$-covering}

\begin{dfn}[{\cite[Definition 2.1]{Asa-cover}}]
\label{dfn:orbit-category}The orbit category $\mathcal{C}/G$ of
$\mathcal{C}$ by $G$ is a category defined as follows.
\begin{itemize}
\item $(\mathcal{C}/G)_0:=\mathcal{C}_0.$
\item For each $x,y\in\mathcal{C}/G$, $(\mathcal{C}/G)(x,y)$ is the set
of all $f=(f_{b,a})\in\prod_{(a,b)\in G\times G}\mathcal{C}(A_{a}x,A_{b}y)$
$ $ such that $f$ is row finite and column finite and that $f_{cb,ca}=A_{c}f_{b,a}$
for all $c\in G$. 
\item For any pair $f\colon x\to y$ and $g\colon y\to z$ in $\mathcal{C}/G$,
$gf:=\left(\sum_{c\in G}g_{b,c}f_{c,a}\right)_{(a,b)}$ .
\end{itemize}
Then $\mathcal{C}/G$ becomes a category where the identity $\id_{x}$
of each $x\in\mathcal{C}/G$ is given by $\id_{x}=(\delta_{a,b}\id_{A_{a}x})_{(a,b)}$.

\end{dfn}
{}
\begin{dfn}[{\cite[Definition 2.4]{Asa-cover}}]
We define a functor $P_{\mathcal{C},G}:=P\colon\mathcal{C}\to\mathcal{C}/G$
as follows.
\begin{itemize}
\item For each $x\in\mathcal{C}$, $P(x):=x$;
\item For each morphism $f$ in $\mathcal{C}$, $P(f):=(\delta_{a,b}A_{a}f)_{(a,b)}$.
\end{itemize}
Then $P$ turns out to be a functor.

\end{dfn}
{}
\begin{dfn}[{\cite[Definition 2.5]{Asa-cover}}]
For each $c\in G$ and $x\in\mathcal{C}$, set $\ps_{c}x:=(\delta_{a,bc}\id_{A_{a}x})_{(a,b)}\in(\mathcal{C}/G)(Px,PA_{c}x)$.
Then $\ps_{c}:=(\ps_{c}x)_{x\in\mathcal{C}}\colon P\to PA_{c}$ is
a natural isomorphism, and the pair 
$(P_{\mathcal{C},G},\ps_{\mathcal{C},G}):=(P,\ps)\colon\mathcal{C}\to\mathcal{C}/G$
turns out to be a $G$-invariant functor, where we set 
 $\ps_{\mathcal{C},G}:=\ps:=(\ps_{c})_{c\in G}$.
We call $(P,\ps)$ the \emph{canonical} functor.
\end{dfn}

\begin{prp}[{\cite[Proposition 2.6]{Asa-cover}}]
\label{prp:univ-orbit}The following statements hold:
\begin{enumerate}
\item $(P,\ps)$ is a $G$-covering functor;
\item $(P,\ps)$ is \emph{universal} among $G$-invariant functors from
$\mathcal{C}$, i.e., for any $G$-invariant functor $(F,\phi)\colon\mathcal{C}\to\mathcal{B}$
there exists a unique functor $H\colon\mathcal{C}/G\to\mathcal{B}$
such that $(F,\phi)=H(P,\psi)$ as $G$-invariant functors.
\end{enumerate}
\end{prp}

\begin{cor}[{\cite[Corollary 2.7]{Asa-cover}}]
\label{cor:2-univ-P}
In the above, $(P,\ps)$ is $2$-universal, i.e., the induced functor
\[
(P,\ps)^{*}\colon\Fun(\mathcal{C}/G,\mathcal{B})\to\Inv(\mathcal{C},\mathcal{B})
\]
is an isomorphism of categories, where $\Fun(\mathcal{C}/G,\mathcal{B})$
is the category of $\k$-functors from $\mathcal{C}/G$ to $\mathcal{B}$.
\end{cor}
This will be used later in \S \ref{dfn:2-fun-?/G}.

\begin{lem}[{\cite[Lemma 5.4]{Asa-cover}}]
$\calC/G$ is $G$-graded.
\end{lem}

Recall the definition of $G$-grading of $\calC/G$:
Let $(P,\psi)\colon\mathcal{C}\to\mathcal{C}/G$ be the canonical functor.
Then the $G$-grading is given by
$(\calC/G)(x,y) = \Ds_{a\in G}(\calC/G)^a(x,y)$,
where
\begin{equation}
\label{eq:grading-C/G}
(\calC/G)^a(x,y):= P_{x,y}^{(1)}(\calC(A_ax, y))
\end{equation}
for all $x, y \in \calC$ and $a \in G$.
Further \cite[Remark 5.5]{Asa-cover} says that
for each $x, y \in \calC$, $a \in G$, and
$f \in (\calC/G)(x, y)$ we have $f \in (\calC/G)^a(x, y)$
if and only if $f_{c,b} = 0$ whenever $c\inv b \ne a$.

\begin{rmk}
In Corollary \ref{cor:2-univ-P} if both $\calC$ and $\calB$ are small categories, then
the corollary above gives us an isomorphism of categories
\[
(P,\ps)^{*}\colon\kCat(\mathcal{C}/G,\mathcal{B})\to\Cat(\mathcal{C},\De\mathcal{B}).
\]
In Lemma \ref{2-fun-/G} we will define a 2-functor $?/G \colon \Cat \to \GrCat$.
If we consider the composite 2-functor $\Fgt \circ (?/G) \colon \Cat \to \kCat$,
where $\Fgt \colon \GrCat \to \kCat$ is the forgetful functor,
we see that the isomorphism above is 2-natural in $\calC$ and in $\calB$.
This means that $\Fgt \circ (?/G)$ is a left adjoint to $\De$.
\end{rmk}

\subsection{Characterization of $G$-covering functors}

The following gives a characterization of $G$-covering functors.

\begin{thm}[{\cite[Theorem 2.9]{Asa-cover}}]
Let $(F,\phi)\colon\mathcal{C}\to\mathcal{B}$ be a $G$-invariant
functor. Then the following are equivalent.
\begin{enumerate}
\item $(F,\phi)$ is a $G$-covering;
\item $(F,\phi)$ is a $G$-precovering that is universal among $G$-precovering
from $\mathcal{C}$;
\item $(F,\phi)$ is universal among $G$-invariant functors from $\mathcal{C}$;
\item There exists an equivalence $H\colon\mathcal{C}/G\to\mathcal{B}$
such that $(F,\phi)\cong H(P,\psi)$ as $G$-invariant functors; and
\item There exists an equivalence $H\colon\mathcal{C}/G\to\mathcal{B}$
such that $(F,\phi)=H(P,\psi)$.
\end{enumerate}
\end{thm}

\subsection{Other isomorphic forms of orbit categories}

The orbit category constructed in Definition \ref{dfn:orbit-category}
has the form of a {}``subset of the product'', which seems not to
match its universality, but it is essentially a left-right symmetrized
direct sum as stated below. (Note that the direct sum of modules were
also constructed as a {}``subset of the direct product''.)
\begin{dfn}[Cibils-Marcos, Keller]
(1) An orbit category $\C\fstorbit G$ is defined as follows.
\begin{itemize}
\item $(\C\fstorbit G)_0:=\C_0$;
\item For any $x,y\in G$, $\C\fstorbit G(x,y):=\Ds_{\al\in G}\C(\al x,y)$;
and
\item For any $x\ya{f}y\ya{g}z$ in $\C\fstorbit G$, $gf:=(\sum_{\al,\be\in G;\be\al=\mu}g_{\be}\cdot\be(f_{\al}))_{\mu\in G}$.
\end{itemize}
(2) Similarly another orbit category $\C\sndorbit G$ is defined as
follows.
\begin{itemize}
\item $(\C\sndorbit G)_0:=\C_0$;
\item For any $x,y\in G$, $(\C\sndorbit G)(x,y):=\Ds_{\be\in G}\C(x,\be y)$;
and
\item For any $x\ya{f}y\ya{g}z$ in $\C\sndorbit G$, $gf:=(\sum_{\al,\be\in G;\al\be=\mu}\al(g_{\be})\cdot f_{\al})_{\mu\in G}$.
\end{itemize}
\end{dfn}
Note that $\C\sndorbit G=(\C\op\fstorbit G)\op$.
\begin{prp}[{\cite[Proposition 2.11]{Asa-cover}}]
\label{prp:three-orbit-cats} We have isomorphisms of categories
$\C\fstorbit G\iso\C/G\iso\C\sndorbit G$.\qed
\end{prp}

\subsection{Composition of a $G$-equivariant functor and a $G$-invariant functor\label{sub:eqv-inv-is-eqv}}

As a special case of Lemma \ref{lem:eqv-eqv-is-eqv}, the composite
of a $G$-equivariant functor and a $G$-invariant functor can be
made into a $G$-invariant functor as follows.
\begin{lem}
\label{lem:equi-and-inv-is-inv}$(1)$ Let $\xymatrix{\mathcal{C}'\ar[r]^{(E,\rho)} & \mathcal{C}\ar[r]^{(F,\phi)} & \mathcal{B}}
$ be functors with $\mathcal{C},\mathcal{C}'$ $G$-categories, $(E,\rho)$
$G$-equivariant and $(F,\phi)$ $G$-invariant. Then \[
(FE,((F\rho_{a})(\phi_{a}E))_{a\in G})\colon\mathcal{C}'\to\mathcal{B}\]
 is a $G$-invariant functor, which we define to be the composite
$(F,\phi)(E,\rho)$ of $(E,\rho)$ and $(F,\phi)$.

$(2)$ In the above if $(E,\rho)$ is a $G$-equivariant equivalence
and $(F,\phi)$ is a $G$-covering functor, then the composite $(F,\phi)(E,\rho)$
is a $G$-covering functor, and hence $\mathcal{C}'/G$ is equivalent
to $\mathcal{B}$.\end{lem}
\begin{proof}
(1) This follows from Lemma \ref{lem:eqv-eqv-is-eqv}.

(2) This is shown in the proof of \cite[Lemma 4.10]{Asa-cover}.
\end{proof}

\section{Smash products}

In this section we cite necessary definitions and statements from
\cite[\S 5]{Asa-cover} without proofs.

\begin{dfn}[{\cite[Definition 5.2]{Asa-cover}}]
Let $\mathcal{B}$ be a $G$-graded category. Then the \emph{smash
product} $\mathcal{B}\#G$ is a category defined as follows.
\begin{itemize}
\item $(\mathcal{B}\#G)_0:=\mathcal{B}_0\times G$, we set $x^{(a)}:=(x,a)$
for all $x\in\mathcal{B}$ and $a\in G$.
\item $(\mathcal{B}\#G)(x^{(a)},y^{(b)}):=\mathcal{B}^{b^{-1}a}(x,y)$ for
all $x^{(a)},y^{(b)}\in\mathcal{B}\#G$.
\item For any $x^{(a)},y^{(b)},z^{(c)}\in\mathcal{B}\#G$ the composition
is given by the following commutative diagram\[
\begin{CD}(\mathcal{B}\#G)(y^{(b)},z^{(c)})\times(\mathcal{B}\#G)(x^{(a)},y^{(b)})@>>>(\mathcal{B}\#G)(x^{(a)},z^{(c)})\\
@\vert@\vert\\
\mathcal{B}^{c^{-1}b}(y,z)\times\mathcal{B}^{b^{-1}a}(x,y)@>>>\mathcal{B}^{c^{-1}a}(x,z),\end{CD}\]
where the lower horizontal homomorphism is given by the composition of
$\mathcal{B}$.
\end{itemize}
\end{dfn}

\begin{lem}[The first part of {\cite[Proposition 5.6]{Asa-cover}}]
$\mathcal{B}\#G$ has a free $G$-action.
\end{lem}
Recall the definition of the free $G$-action on $\mathcal{B}\#G$:
For each $c\in G$ and $x^{(a)}\in\mathcal{B}\#G$, $A_{c}x^{(a)}:=x^{(ca)}$.
For each $f\in(\mathcal{B}\#G)(x^{(a)},y^{(b)})=\mathcal{B}^{b^{-1}a}(x,y)=(\mathcal{B}\#G)(x^{(ca)},y^{(cb)})$,
$A_{c}f:=f$.
\begin{dfn}[{\cite[Definition 5.7]{Asa-cover}}]
Let $\mathcal{B}$ be a $G$-graded category. Then we define a functor
$Q_{\mathcal{B},G}:=Q\colon\mathcal{B}\#G\to\mathcal{B}$ as follows. 
\begin{itemize}
\item $Q(x^{(a)})=x$ for all $x^{(a)}\in\mathcal{B}\#G$ .
\item $Q(f):=f$ for all $f\in(\mathcal{B}\#G)(x^{(a)},y^{(b)})=\mathcal{B}^{b^{-1}a}(x,y)$
.
\end{itemize}
\end{dfn}

\begin{prp}[{\cite[Proposition 5.8, Remark 5.9]{Asa-cover}}]
\label{prp:cov-fun-wrt-smash}$Q=QA_{a}$ for all $a\in G$ and $Q=(Q,\id)\colon\mathcal{B}\#G\to\mathcal{B}$
is a $G$-covering functor. Hence in particular, $Q$ factors through
the canonical $G$-covering functor $(P,\psi)\colon\mathcal{B}\#G\to(\mathcal{B}\#G)/G$,
i.e., there exists a unique equivalence $H\colon(\mathcal{B}\#G)/G\to\mathcal{B}$
such that $Q=H(P,\psi)$.
\end{prp}

\section{2-functors}
\subsection{Orbit 2-functor}

We first extend the orbit category construction to a 2-functor
$\Cat \to \GrCat$.
\begin{dfn}\label{dfn:2-fun-?/G}
Let $(E,\rho),(E',\rho')\colon\mathcal{C}\to\mathcal{C}'$
be 1-morphisms and $\eta\colon(E,\rho)\to(E',\rho')$ a 2-morphism in $\Cat$.
Set $(P,\psi)\colon\mathcal{C}\to\mathcal{C}/G$, $(P',\psi')\colon\mathcal{C}'\to\mathcal{C}'/G$
to be the canonical functors. By Proposition \ref{lem:equi-and-inv-is-inv}
we have $(P',\psi')\eta\colon(P',\psi')(E,\rho)\to(P',\psi')(E',\rho')$
is in $\Inv(\mathcal{C},\mathcal{C}'/G)$. Then using the isomorphism
$(P,\psi)^{*}\colon\Fun(\mathcal{C}/G,\mathcal{C}'/G)\to\Inv(\mathcal{C},\mathcal{C}'/G)$
of categories we can define
\begin{eqnarray*}
(E,\rho)/G & := & {(P,\psi)^{*}}^{-1}((P',\psi')(E,\rho))\text{ and}\\
\eta/G & := & {(P,\psi)^{*}}^{-1}((P'\psi')\eta).
\end{eqnarray*}
This construction is visualized in the following diagram: \[ 
\xymatrix@C=9em@R=9em{
\calC \rtwocell^{(E,\rho)}_{(E',\rho')}{\eta}& \calC'\\
\calC/G \rtwocell^{(E,\rho)/G}_{(E',\rho')/G}{\ \ \ \eta/G} & \calC'/G.
\ar_{(P,\psi)} "1,1";"2,1" 
\ar^{(P',\psi')} "1,2";"2,2" 
\ar@/^1.5pc/^{(P',\psi')(E,\rho)} "1,1";"2,2" \ar@/_1.4pc/_{(P',\psi')(E',\rho')} "1,1";"2,2" \ar@{=>}^{(P',\psi')\eta} "1,2"+<-5.6em,-4.9em>;"2,1"+<5.2em,4.8em> 
}
\] The explicit form of $\eta/G$ is given by\[
(\eta/G)Px:=P'(\eta x)\in(\mathcal{C}'/G)^{1}(((E,\rho)/G)Px,((E',\rho')/G)Px)\]
(for $(\mathcal{C}'/G)^{1}$ see \eqref{eq:grading-C/G}) for all $x\in\mathcal{C}$. Then as easily seen, $(E,\rho)/G$ is
a strictly degree-preserving functor and $\eta/G$ is a 2-morphism in
$\GrCat$.
\end{dfn}

\begin{lem}\label{2-fun-/G}
The definition above extends the orbit category construction to a
$2$-functor \[
?/G\colon G\text{-}\mathbf{Cat}\to G\text{-}\mathbf{GrCat}.\]
\end{lem}
\begin{proof}
(1) $\id_{\mathcal{C}}/G=\id_{\mathcal{C}/G}$ for all $\mathcal{C}\in\Cat$. 

Indeed, let $(P,\psi)\colon\mathcal{C}\to\mathcal{C}/G$ be the canonical
functor. Then this follows from the following strict commutative diagram:\[
\xymatrix{\calC\ar[r]^{\id_{\calC}}\ar[d]_{(P,\ps)} & \calC\ar[d]^{(P,\ps)}\\
\calC/G\ar[r]_{\id_{\calC/G}} & \calC/G.}
\]

(2) For any $\mathcal{C}\xrightarrow{(E,\rho)}\mathcal{C}'\xrightarrow{(E',\rho')}\mathcal{C}''$
in $\Cat$, $((E',\rho')\cdot(E,\rho))/G=(E',\rho')/G\cdot(E,\rho)/G$.

Indeed, let $(P,\psi)\colon\mathcal{C}\to\mathcal{C}/G$, $(P',\psi')\colon\mathcal{C}'\to\mathcal{C}'/G$,
$(P'',\psi'')\colon\mathcal{C}''\to\mathcal{C}''/G$ be the canonical
functor. We can set $(E,\rho)/G=(H,1)\colon\mathcal{C}/G\to\mathcal{C}'/G$
and $(E',\rho')/G=(H',1)\colon\mathcal{C}'/G\to\mathcal{C}''/G$.
Then we have the following strictly commutative diagram consisting
of solid arrows:$$
\xymatrix{
\calC  & \calC'  & \calC''\\
\calC/G  & \calC'/G & \calC''/G.
\ar_{(E,\rho)} "1,1"; "1,2" 
\ar_{(E',\rho')} "1,2"; "1,3" 
\ar_{(P,\ps)} "1,1"; "2,1" 
\ar^{(P', \ps')} "1,2"; "2,2" 
\ar^{(P'', \ps'')} "1,3"; "2,3" 
\ar^{H} "2,1"; "2,2" 
\ar^{H'} "2,2"; "2,3" 
\ar@{..>}@/^1pc/^{(E'E,\rho'')} "1,1"; "1,3" 
\ar@{..>}@/_1pc/_{H'H} "2,1"; "2,3"
}
$$
Comparing the second entries of $G$-invariant functors this implies the following for
all $a\in G$:
\begin{eqnarray}
(P'\rho_{a})(\psi'_{a}E) & = & H\psi_{a}\label{eq:left-sq}\\
(P''\rho'_{a})(\psi''_{a}E') & = & H'\psi'_{a}\label{eq:right-sq}
\end{eqnarray}
Set $(E'E,\rho''):=(E',\rho')\cdot(E,\rho)$, namely, $\rho'':=((E'\rho_{a})(\rho'_{a}E))_{a\in G}$.
Then the two triangles consisting of dotted arrows and horizontal
arrows are strictly commutative. This shows the strict commutativity
of the following as a diagram of functors:\begin{equation}
\xymatrix{\calC & \calC''\\
\calC/G & \calC''/G,\ar_{(P,\ps)}"1,1";"2,1"\ar^{(P'',\ps'')}"1,2";"2,2"\ar^{(E'E,\rho'')}"1,1";"1,2"\ar_{H'H}"2,1";"2,2"}
\label{eq:comm-outer-sq}\end{equation}
i.e., we have $P''E'E=H'HP$. We have to verify that this is strictly
commutative as a diagram of $G$-invariant functors, i.e., that the
following holds: \[
(P'',\psi'')\cdot(E'E,\rho'')=H'H\cdot(P,\psi).\]
Looking at the second entries of $G$-invariant functors it is enough to show the following
for all $a\in G$:\begin{equation}
(P''\rho''_{a})(\psi''_{a}E'E)=H'H\psi_{a}.\label{eq:outer-sq}\end{equation}
From \eqref{eq:left-sq} the composition with $H'$ on the left yields\[
(H'P'\rho_{a})(H'\psi'_{a}E)=H'H\psi_{a}.\]
From \eqref{eq:right-sq} the composition with $E$ on the right yields\[
(P''\rho'_{a}E)(\psi''_{a}E'E)=H'\psi'_{a}E.\]
Using these equalities we see that the left hand side of \eqref{eq:outer-sq}
is equal to\begin{eqnarray*}
(P''E'\rho_{a})(P''\rho'_{a}E)(\psi'_{a}E'E) & = & (P''E'\rho'_{a})(H'\psi'_{a}E)\\
 & = & (H'P'\rho_{a})(H'\psi'_{a}E)\\
 & = & H'H\psi_{a},\end{eqnarray*}
the right hand side, and the strict commutativity of \eqref{eq:comm-outer-sq}
as a diagram of $G$-invariant functors is verified, which shows that
$((E',\rho')(E,\rho))/G=H'H=(E',\rho')/G\cdot(E,\rho)/G$.

(3) $\id_{(E,\rho)}/G=\id_{(E,\rho)/G}$ for all 1-morphism $(E,\rho)\colon\mathcal{C}\to\mathcal{C}'$
in $\Cat$.

Indeed, set $(P,\psi)$, $(P',\psi')$, $H$ to be as in (2) above.
For each $Px\in\mathcal{C}/G$ we have $(\id_{(E,\rho)}/G)(Px)=P'((\id_{(E,\rho)})x)=\id_{P'Ex}=\id_{HPx}=(\id_{(E,\rho)/G})(Px)$.

(4) $?/G$ preserves the vertical composition.

Indeed, let $(E,\rho),(E'\rho'),(E'',\rho'')\in(\Cat)(\mathcal{C},\mathcal{C}')$,
and let $\eta\colon(E,\rho)\To(E',\rho')$, $\eta'\colon(E',\rho')\To(E'',\rho'')$
be 2-morphisms in $\Cat$. Set $(P,\psi)$, $(P',\psi')$ to be as in
(2) above. Then for each $Px\in\mathcal{C}/G$ we have\[
((\eta'\eta)/G)(Px)=P'((\eta'\eta)x)=P'(\eta'x)P'(\eta x)=(\eta'/G)(Px)\cdot(\eta/G)(Px).\]
This shows that $(\eta'\eta)/G=(\eta'/G)(\eta/G)$.

(5) $?/G$ preserves the horizontal composition.

Indeed, let $(E,\rho),(E'\rho')\in(\Cat)(\mathcal{C},\mathcal{C}')$,
$(F,\tau),(F'\tau')\in(\Cat)(\mathcal{C}',\mathcal{C}'')$ and $\eta\colon(E,\rho)\To(E',\rho')$,
$\eta'\colon(F,\tau)\To (F',\tau')$ be 2-morphisms in $\Cat$. Then we
have to show the equality\[
(\eta'*\eta)/G=(\eta'/G)*(\eta/G).\]
Set $(P,\psi)$, $(P',\psi')$ and $(P'',\psi'')$ to be as in (2)
above. Then for each $Px\in\mathcal{C}/G$ we have \begin{eqnarray*}
((\eta'*\eta)/G)(Px) & = & P''((\eta'*\eta)x)=P''(((F'\eta)(\eta'E))x)=P''((F'\eta)x\cdot(\eta'E)x)\\
 & = & P''((F'\eta)x)P''((\eta'E)x)=P''(F'(\eta x))\cdot P''(\eta'(Ex)),\end{eqnarray*}
and\begin{eqnarray*}
((\eta'/G)*(\eta/G))(Px) & = & ((F',\tau')/G\cdot\eta/G)(Px)\cdot(\eta'/G\cdot(e,\rho)/G)(Px)\\
 & = & ((F',\tau')/G)(P'(\eta x))\cdot(\eta'/G)(P'Ex)\\
 & = & P''(F'(\eta x))\cdot P''(\eta'(Ex)),\end{eqnarray*}
from which the equality follows, where $((F',\tau')/G)(P'(\eta x))=P''(F'(\eta x))$
follows from the commutative diagram\[
\xymatrix{\calC'(Ex,E'x)\ar@{^{(}->}[r]\ar[rdd]_{P'} & \Ds_{a\in G}\calC'(A_{a}Ex,E'x)\ar[r]^{F'}\ar[dd]_{P'{}_{Ex,E'x}^{(1)}} & \Ds_{a\in G}\calC''(F'A_{a}Ex,F'E'x)\ar[d]^{\Ds_{a\in G}\calC''(\ta_{a}Ex,F'E'x)}\\
 &  & \Ds_{a\in G}\calC''(A_{a}F'Ex,F'E'x)\ar[d]^{P''{}_{F'Ex,F'E'x}^{(1)}}\\
 & (\calC/G)(P'Ex,P'E'x)\ar[r]_{(F'\ta')/G} & (\calC''/G)(P''F'Ex,P''F'E'x).}
\]

As a consequence, $?/G\colon\Cat\to\GrCat$ is a 2-functor. 
\end{proof}

\subsection{Smash 2-functor}

Next we extend the smash product construction to a 2-functor.
\begin{dfn}\label{2-smash}
Let $(H,r)\colon\mathcal{B}\to\mathcal{B}'$ be in $\GrCat$. Then
the functor $(H,r)\#G\colon\mathcal{B}\#G\to\mathcal{B}'\#G$ is defined
as follows.\textbf{}\\
\textbf{On objects}. For each $x^{(a)}\in\mathcal{B}\#G$ we set\[
((H,r)\#G)(x^{(a)}):=(Hx)^{(ar_{x})}.\]
\textbf{On morphisms}. For each $f\in(\mathcal{B}\#G)(x^{(a)},y^{(b)})=\mathcal{B}^{b^{-1}a}(x,y)$
we set\[
((H,r)\#G)(f):=H(f),\]
which is an element of $\mathcal{B}'^{r_{y}^{-1}b^{-1}ar_{x}}(Hx,Hy)=(\mathcal{B}'\#G)((Hx)^{(ar_{x})},(Hy)^{(br_{y})})$.
Then as easily seen, $(H,r)\#G$ is a strictly $G$-equivariant functor,
and hence $(H,r)\#G=((H,r)\#G,1)\colon\mathcal{B}\#G\to\mathcal{B}'\#G$
is in $\Cat$.

Next let $(H',r')\colon\mathcal{B}\to\mathcal{B}'$ be a 1-morphism and
$\theta\colon(H,r)\to(H',r')$ a 2-morphism in $\GrCat$.
We define $\theta\#G\colon(H,r)\#G\To(H',r')\#G$
by\[
(\theta\#G)x^{(a)}:=\theta x\]
for all $x^{(a)}\in\mathcal{B}\#G$. Then it is easy to see that $\theta\#G$
is a 2-morphism in $\Cat$. 
\end{dfn}

\begin{lem}
The definition above extends the smash product construction to a $2$-functor
\[
?\#G\colon G\text{-}\mathbf{GrCat}\to G\text{-}\mathbf{Cat}.\]
\end{lem}
\begin{proof}
We only show that $?\#G$ preserves the horizontal composition because
the other properties for $?\#G$ to be a 2-functor are immediate from
the definition. Let $(H,\xi),(H',\xi')\in\GrCat(\mathcal{B},\mathcal{B}')$,
$(F,\ze),(F'\ze')\in\GrCat(\mathcal{B}',\mathcal{B}'')$ and let $\th\colon(H,\xi)\To (H',\xi')$,
$\th'\colon(F,\ze)\To (F'\ze')$ be 2-morphisms in $\GrCat$. For each
$x^{(a)}\in\mathcal{B}\#G$ we have\[
((\theta'*\theta)\#G)(x^{(a)})=(\th'*\th)x=(F'\th)x\cdot(\th'H)x=F'(\th x)\cdot\th'(Hx),\]
and\begin{eqnarray*}
((\th\#G)*(\th\#G))(x^{(a)}) & = & (((F',\ze')\#G)(\th\#G))((\th'\#G)((H,\xi)\#G)))(x^{(a)})\\
 & = & ((F'\ze')\#G)(\th\#G))(x^{(a)})\cdot((\th'\#G)((H,\xi)\#G))(x^{(a)})\\
 & = & ((F',\ze')\#G)(\th x)\cdot(\th'\#G)((Hx)^{(a\xi_{x})})\\
 & = & F'(\th x)\cdot\th'(Hx),\end{eqnarray*}
which shows that $(\theta'*\theta)\#G=(\th\#G)*(\th\#G)$.
\end{proof}

\subsection{Main theorem}

We are now in a position to state our main result, which is a precise form
of Theorem \ref{thm-2-eq}.

\begin{thm}\label{main-thm}
Both $2$-functors $?/G$ and $?\#G$ are $2$-equivalences.
They are mutual $2$-quasi-inverses. Hence the $2$-categories $G$-$\mathbf{Cat}$
and $G$-$\mathbf{GrCat}$ are $2$-equivalent. More precisely, we have
four $2$-natural isomorphisms
\begin{eqnarray*}
\varepsilon\colon\id_{G\text{-}\mathbf{Cat}} & \To & (?\#G)(?/G)\\
\varepsilon'\colon(?\#G)(?/G) & \To & \id_{G\text{-}\mathbf{Cat}}\\
\omega\colon\id_{G\text{-}\mathbf{GrCat}} & \To & (?/G)(?\#G)\\
\omega'\colon(?/G)(?\#G) & \To & \id_{G\text{-}GrCat}
\end{eqnarray*}
with the property that\begin{eqnarray}
\varepsilon'_{\mathcal{C}}\varepsilon_{\mathcal{C}} & = & \id_{\mathcal{C}},\label{eq:ep'ep}\\
\varepsilon{}_{\mathcal{C}}\varepsilon'_{\mathcal{C}} & \cong & \id_{(\mathcal{C}/G)\#G},\label{eq:epep'}\\
\omega'_{\mathcal{B}}\omega_{\mathcal{B}} & = & \id_{\mathcal{B}},\label{eq:om'om}\\
\omega{}_{\mathcal{B}}\omega'_{\mathcal{B}} & \cong & \id_{(\mathcal{B}\#G)/G},\label{eq:omom'}\end{eqnarray}
and that $\varepsilon'_{\mathcal{C}}$ are strictly $G$-equivariant
functors and $\omega_{\mathcal{B}}$ are strictly degree-preserving
functors for all $\mathcal{C}\in G\text{-}\mathbf{Cat}$ and $\mathcal{B}\in G\text{-}\mathbf{GrCat}$.
Furthermore $\varepsilon$ and $\omega'$ are strictly $2$-natural transformations,
and in particular, $?/G$ is strictly left $2$-adjoint to $?\#G$. Namely
the pasting of the diagram\begin{equation}
\vcenter{\xymatrix{G\text{-}\mathbf{GrCat} &  & G\text{-}\mathbf{GrCat}\\
 & G\text{-}\mathbf{Cat} &  & G\text{-}\mathbf{Cat}\ar^{\id}"1,1";"1,3"\ar_{?\#G}"1,1";"2,2"\ar^{?\#G}"1,3";"2,4"\ar_{\id}"2,2";"2,4"\ar_{?/G}"2,2";"1,3"\ar_{\omega'}@{=>}"2,2";"1,2"\ar_{\varepsilon}@{=>}"2,3";"1,3"}
}\label{eq:sharp-G}\end{equation}
is equal to the identity $\id_{?\#G}$, and the pasting of the diagram\begin{equation}
\vcenter{\xymatrix{ & G\text{-}\mathbf{GrCat} &  & G\text{-}\mathbf{GrCat}\\
G\text{-}\mathbf{Cat} &  & G\text{-}\mathbf{Cat}\ar^{?/G}"2,1";"1,2"\ar^{?\#G}"1,2";"2,3"\ar_{?/G}"2,3";"1,4"\ar_{\id}"2,1";"2,3"\ar^{\id}"1,2";"1,4"\ar_{\varepsilon}@{=>}"2,2";"1,2"\ar_{\omega'}@{=>}"2,3";"1,3"}
}\label{eq:slash-G}\end{equation}
is equal to the identity $\id_{?/G}$.
\end{thm}
The proof is given in the next section.

\subsection{Proof of Theorem \ref{thm-2-mor}}
(1) and (2) These are direct consequences of \eqref{eq:ep'ep}--\eqref{eq:omom'}.

(3) This follows from \eqref{eq:sharp-G} and \eqref{eq:slash-G}
by a general theory of 2-categories (see e.g.\ \cite{Ke-St}, \cite{Bor};
the proof proceeds just the same way as in the usual category case).

(4) $\Cat(\calC, \calC') \simeq \Cat(\calC, (\calC'/G)\# G)
\iso \Cat(\calC/G, \calC'/G)$.

(5) A similar proof as above works.
\qed

\medskip
Theorem \ref{thm-2-mor} gives the following.
\begin{cor}
Let $\calC, \calC' \in \Cat$.
Then we have a faithful embedding
$$
\Cat(\calC, \calC') \to \Inv(\calC, \calC'/G)
$$
of $\k$-categories.
\end{cor}

\begin{proof}
$\Cat(\calC, \calC') \simeq \GrCat(\calC/G, \calC'/G)
\subseteq \Fun(\calC/G, \calC'/G) \iso \Inv(\calC, \calC'/G)$,
where the first equivalence is an injection on objects
by \eqref{eq:ep'ep}.
Indeed, if $(F, \ph), (F', \ph') \in \Cat(\calC, \calC')$ and
$(F, \ph)/G = (F', \ph')/G$, then the naturality of $\ep$ shows that
$$
\ep_{\calC'}(F,\ph) = (((F, \ph)/G)\# G) \ep_\calC =
(((F', \ph')/G)\# G) \ep_\calC = \ep_{\calC'}(F',\ph').
$$
Hence by \eqref{eq:ep'ep} we have $(F, \ph) = (F', \ph')$.
\end{proof}

\subsection{Weak universality of the canonical functor of a smash product}

As an application of Theorem \ref{main-thm} we obtain the proposition below,
which states that the canonical functor
$(Q,\id)\colon\mathcal{B}\#G\to\mathcal{B}$ to a $G$-graded category $\calB$ has the weak universality
among $G$-invariant functors from $G$-categories to $\calB$ that
{\em induce degree-preserving functors}
(see Definition \ref{dfn:induce-degree-preserving} below).
(It often does not have the universality as Remark \ref{not-universal} shows.)

\begin{dfn}
\label{dfn:induce-degree-preserving}
Let $\calC$ be a $G$-category with the canonical functor
$(P,\psi)\colon\mathcal{C}\to\mathcal{C}/G$,
$\calB$ a $G$-graded category,
and $r \colon \calC_0 \to G$ a map.
Then a $G$-invariant functor $(F,\phi)\colon\mathcal{C}\to\mathcal{B}$
is said to {\em induce a degree-preserving functor with} $r$ if
the unique functor $H\colon\mathcal{C}/G\to\mathcal{B}$
such that $(F,\phi)=H(P,\psi)$ (the existence of which is guaranteed 
by Proposition $\ref{prp:univ-orbit}$) has the property that
$(H, r)$ is a degree-preserving functor.
\end{dfn}

\begin{lem}
\label{lem:compatible-degree-preserving}
Let $\calC$ be a $G$-category and $\calB$ a $G$-graded category.
Then a $G$-invariant functor $(F,\phi)\colon\mathcal{C}\to\mathcal{B}$
induces a degree-preserving functor with a map $r\colon \calC_0 \to G$
if and only if for each $x, y \in \calC$ and $a \in G$
the restriction of
$$
F_{x,y}^{(1)}\colon\Ds_{b\in G}\calC(A_{b}x,y) \to \calB(Fx, Fy)
$$
to $\calC(A_{r_ya}x,y)$ induces a homomorphism
$\calC(A_{r_ya}x,y) \to \calB^{ar_x}(Fx, Fy)$,
or equivalently,  for each $f \in \calC(A_{r_ya}x,y)$ we have
$F(f)\cdot \ph_{r_ya}x \in \calB^{ar_x}(Fx, Fy)$.
\end{lem}

\begin{proof}
This follows from the definition \eqref{eq:grading-C/G}
of the $G$-grading of $\calC/G$ and the commutativity of the diagram
\begin{equation}
\label{eq:induced-funct}
\xymatrix{
\Ds_{a\in G}\calC(A_a x,y) \ar[rr]^{F^{(1)}_{x,y}}
\ar[rd]_{P^{(1)}_{x,y}}&&\calB(Fx, Fy)\\
&(\calC/G)(x,y).\ar[ru]_H
}
\end{equation}
(see Proof of \cite[Proposition 2.6 (3)]{Asa-cover}).
\end{proof}

\begin{prp}
\label{prp:smash-prod-fun-weak-univ}
Let $\calC$ be a $G$-category, $\mathcal{B}$ a $G$-graded
category, and $(Q,\id)\colon\mathcal{B}\#G\to\mathcal{B}$ the canonical
functor. If $(F,\phi)\colon\mathcal{C}\to\mathcal{B}$
is a $G$-invariant functor inducing a degree-preserving functor, then
there exists a $G$-equivariant functor $(K,\rho)\colon\mathcal{C}\to\mathcal{B}\#G$
such that $(F, \phi)=(Q,\id)(K,\rho)$.
\end{prp}

\begin{proof}
Let $(P,\psi)\colon\mathcal{C}\to\mathcal{C}/G$ be the canonical
functor, and assume that a $G$-invariant functor
$(F,\phi)\colon\mathcal{C}\to\mathcal{B}$
induces a degree-preserving functor with a map $r\colon \calC_0 \to G$.
Then there exists a unique equivalence $H\colon\mathcal{C}/G\to\mathcal{B}$
such that $(F,\phi)=H(P,\psi)$ and $(H, r)$ is a degree-preserving functor.
It is easy to verify the commutativity of the diagram

\[
\xymatrix@C=5em{ & \calC\ar[r]^{(F,\ph)}\ar[d]_{(P,\ps)}\ar[ld]_{(\ep_{\calC},\phi_{\mathcal{C}})} & \calB\ar@{=}[d]\\
(\calC/G)\#G\ar[r]_{(Q_{\calC/G},\id)}\ar[rd]_{(H, r)\#G} & \calC/G\ar[r]_{(H, r)} & \calB\\
 & \calB\#G\ar[ur]_{(Q,\id)}}
\]
using the explicit forms of the functors (see Definition \ref{dfn:ep}
for $\ep_{\mathcal{C}}=(\ep_{\mathcal{C}},\ph_{\mathcal{C}})$). Thus
we can take $(K,\rho):=((H, r)\#G)(\ep_{\mathcal{C}},\ph_{\mathcal{C}})$,
which is $G$-equivariant by Lemma \ref{lem:eqv-eqv-is-eqv}.\end{proof}
\begin{rmk}\label{not-universal}
In the above proposition $(K,\rho)$ is not uniquely determined in
general. For instance, consider the case that the center $Z(G$) of
$G$ is not trivial, and take $\mathcal{C}:=\mathcal{B}\#G$ and $(F,\phi):=(Q,\id)$.
Then $(K,\rho):=(A_{a},\id)$ satisfies the required property for
all $a\in Z(G)$.
\end{rmk}

Also the weak universality of $(Q,\id)\colon\mathcal{B}\#G\to\mathcal{B}$
gives us a characterization of a $G$-covering functor to $\calB$
inducing a degree-preserving functor.

\begin{prp}
\label{prp:smash-prod-fun}
Let $\calC$ be a $G$-category, $\mathcal{B}$ a $G$-graded category with
the canonical functor $(Q,\id)\colon\mathcal{B}\#G\to\mathcal{B}$,
and $(F,\phi)\colon\mathcal{C}\to\mathcal{B}$ a $G$-invariant functor
inducing a degree-preserving functor.
Then $(F,\phi)$ is a $G$-covering functor if and only if there exists
a $G$-equivariant equivalence $(K,\rho)\colon\mathcal{C}\to\mathcal{B}\#G$
such that $(F,\phi)=(Q,\id)(K,\rho)$.
\end{prp}

\begin{proof}
($\implies$). We keep the notation and the argument used in the proof
of the proposition above, which constructed a $G$-equivariant functor
$(K,\rho)\colon\mathcal{C}\to\mathcal{B}\#G$ such that $(F, \phi)=(Q,\id)(K,\rho)$.
Since $?\#G$ is a 2-functor, $(H, r)\#G$ is an equivalence. In
addition $(\ep_{\mathcal{C}},\ph_{\mathcal{C}})$ is also a $G$-equivariant
equivalence by Theorem \ref{main-thm}. Hence as the composite of
these $(K,\rho)$ is an equivalence.

($\impliedby$). This follows by Lemma \ref{lem:equi-and-inv-is-inv}(2).
\end{proof}

\section{Proof of Theorem \ref{main-thm}.}

\subsection{$\varepsilon\colon\id_{G\text{-}\mathbf{Cat}}\To (?\#G)(?/G)$}

\begin{dfn}[see {\cite[Theorem 5.10]{Asa-cover}}]
\label{dfn:ep}Let $\mathcal{C}$ be an object of $\Cat$ and $(P,\psi)\colon\mathcal{C}\to\mathcal{C}/G$
the canonical functor. We define a $G$-equivariant functor $\varepsilon_{\mathcal{C}}\colon\mathcal{C}\to(\mathcal{C}/G)\#G$
as follows.\textbf{}\\
\textbf{On objects}. For each $x\in\mathcal{C}$ we set\[
\varepsilon_{\mathcal{C}}(x):=(Px)^{(1)}.\]
\textbf{On morphisms}. For each $f\colon x\to y$ in $\mathcal{C}$,
we set\[
\varepsilon_{C}(f):=P_{x,y}^{(1)}(f)\,(=P(f)).\]
\textbf{Natural isomorphisms}. For each $a\in G$ we define a natural
transformation $\phi_{a}\colon A_{a}\varepsilon_{\mathcal{C}}\to\varepsilon_{\mathcal{C}}A_{a}$
by $\phi_{a}x:=\psi_{a}x$ for all $x\in\mathcal{C}$, i.e., by the
commutative diagram\[
\begin{CD}A_{a}\ep_{\calC}x@>{\ph_{a}x}>>\ep_{\mathcal{C}}A_{a}x\\
@\vert@\vert\\
(Px)^{(a)}@>>\psi_{a}x>(PA_{a}x)^{(1)}.\end{CD}\]
Here note that $((\mathcal{C}/G)\#G)((Px)^{(a)},(PA_{a}x)^{(1)})=(\mathcal{C}/G)^{a}(Px,PA_{a}x)\ni\psi_{a}x$.
Set $\phi_{\mathcal{C}}:=(\phi_{a})_{a\in G}$. Then we have already
shown that $\varepsilon_{\mathcal{C}}=(\varepsilon_{\mathcal{C}},\phi_{\mathcal{C}})$
is a $G$-equivariant equivalence in \cite[Theorem 5.10]{Asa-cover}.\end{dfn}
\begin{lem}
$\varepsilon$ is a strictly $2$-natural transformation.\end{lem}
\begin{proof}
Let $\mathcal{C},\mathcal{C}'\in\Cat$.

(1) Let $(E,\rho)\in(\Cat)(\mathcal{C},\mathcal{C}')$. Set $(H,1):=(E,\rho)/G$.
Then we have a strictly commutative diagram\[
\xymatrix{\mathcal{C}\ar[r]^{(E,\rho)}\ar[d]_{(P,\psi)} & \mathcal{C}'\ar[d]^{(P',\psi')}\\
\mathcal{C}/G\ar[r]_{(H,1)} & \mathcal{C}'/G,}
\]
where the vertical arrows are the canonical functors. For each $x,y\in\mathcal{C}$
we have a commutative diagram\[
\xymatrix{\calC(x,y)\ar@{^{(}->}[r] & \Ds_{a\in G}\calC(A_{a}x,y)\ar[r]^{E}\ar[dd]_{P_{x,y}^{(1)}} & \Ds_{a\in G}\calC'(EA_{a}x,Ey)\ar[d]^{\Ds_{a\in G}\calC'(\ro_{a}x,Ey)}\\
 &  & \Ds_{a\in G}\calC'(A_{a}Ex,Ey)\ar[d]^{P'{}_{Ex,Ey}^{(1)}}\\
 & (\calC/G)(Px,Py)\ar[r]_{H_{Px,Py}} & (\calC'/G)(P'Ex,P'Ey)}
\]
by which it is easy to see that the following diagram is strictly
commutative:\[
\xymatrix{\mathcal{C}\ar[r]^{(E,\rho)}\ar[d]_{\ep_{\mathcal{C}}} & \mathcal{C}'\ar[d]^{\ep_{\mathcal{C}'}}\\
(\mathcal{C}/G)\#G\ar[r]_{(H,1)\#G} & (\mathcal{C}'/G)\#G.}
\]

(2) Let $\et\colon(E,\rho)\to(E',\rho')$ be in $(\Cat)(\mathcal{C},\mathcal{C}')$.
Set $(H,1):=(E,\rho)/G$, $(H',1):=(E',\rho')/G$ and $\th\colon=\et/G$.
Then it immediately follows from definition that $\ep_{\mathcal{C}'}\et=(\et/G)\#G\cdot\ep_{\mathcal{C}}$.

By (1) and (2) above $\ep$ is a strictly 2-natural transformation.
\end{proof}

\subsection{\textmd{$\varepsilon'\colon(?\#G)(?/G)\To\id_{G\text{-}\mathbf{Cat}}$}}
\begin{dfn}
Let $\mathcal{C}$ be an object of $\Cat$ and $(P,\psi)\colon\mathcal{C}\to\mathcal{C}/G$
the canonical functor. We define a $G$-equivariant functor $\varepsilon'_{\mathcal{C}}\colon(\mathcal{C}/G)\#G\to\mathcal{C}$
as follows.\textbf{}\\
\textbf{On objects}. For each $x\in\mathcal{C}$ and $a\in G$
we set
\[
\varepsilon_{\mathcal{C}}'((Px)^{(a)}):=A_{a}x.
\]
\textbf{On morphisms}. Let $f\colon(Px)^{(a)}\to(Py)^{(b)}$ be in
$(\mathcal{C}/G)\#G$. Then we have the diagram
\[
\xymatrix{((\mathcal{C}/G)\#G)((Px)^{(a)},(Py)^{(b)}) & \calC(A_{a}x,A_{b}x)\\
(\mathcal{C}/G)^{b^{-1}a}(Px,Py) & \calC(A_{b^{-1}a}x,y).\ar@{-->}"1,1";"1,2"\ar_{\iso}^{P_{x,y}^{(1)}}"2,2";"2,1"\ar_{A_{b}}^{\wr}"2,2";"1,2"\ar@{=}"1,1";"2,1"}
\]
Using this we set\[
\varepsilon'_{\mathcal{C}}(f):=A_{b}{P_{x,y}^{(1)}}^{-1}(f).\]
\textbf{Natural isomorphisms}. For each $a\in G$ we easily see that
$A_{a}\varepsilon'_{\mathcal{C}}=\varepsilon'_{\mathcal{C}}A_{a}$.
Thus $\varepsilon'_{\mathcal{C}}$ is a strictly $G$-equivariant
functor.\end{dfn}
\begin{lem}
$\varepsilon'$ is a $2$-natural transformation.
\end{lem}
\begin{proof}
Let $(E,\rho)\colon\mathcal{C}\to\mathcal{C}'$ be a 1-morphism in $\Cat$.
We define a natural transformation $\Psi_{(E,\rho)}$ in the diagram
\[
\xymatrix@C=7em{(\calC/G)\#G\ar[r]^{((E,\rho)/G)\#G} & (\calC'/G)\#G\\
\calC\ar[r]_{(E,\rho)} & \calC'\ar_{\ep'_{\calC}}"1,1";"2,1"\ar^{\ep'_{\calC'}}"1,2";"2,2"\ar@{=>}_{\Psi_{(E,\rho)}}^{\iso}"1,2";"2,1"}
\]
by\[
(\Psi_{(E,\rho)})(Px)^{(a)}:=\rho_{a}x\]
for all $(Px)^{(a)}\in(\mathcal{C}/G)\#G$. Then it is not hard to
verify that $\Psi_{(E,\rho)}$ is a natural isomorphism. This shows
the 1-naturality of $\ep'$. Now let $(E',\rho)'\colon\mathcal{C}\to\mathcal{C}'$
be another 1-morphism and $\eta\colon(E,\rho)\To (E'\rho')$ a 2-morphism
in $\Cat$. Then it is easy to check the commutativity of the diagram
\[
\xymatrix@C=5em{
\ep'_{\calC'}\cdot((E,\rho)/G)\#G & (E,\rho)\cdot \ep'_{\calC}\\
\ep'_{\calC'}\cdot((E',\rho')/G)\#G & (E',\rho')\cdot \ep'_{\calC}
\ar@{=>}^(0.6){\Psi_{(E,\rho)}}_(0.6){\iso}"1,1";"1,2"
\ar@{=>}_(0.6){\Psi_{(E',\rho')}}^(0.6){\iso}"2,1";"2,2"
\ar@{=>}^{\ep'_{\calC'}\cdot((\et/G)\#G)}"1,1";"2,1"
\ar@{=>}^{\et\cdot\ep'_{\calC}}"1,2";"2,2"
}
\]
of natural transformations, which shows the 2-naturality of $\ep'$.
\end{proof}

\subsection{\textmd{$\omega\colon\id_{G\text{-}\mathbf{GrCat}}\To (?/G)(?\#G)$}}

\begin{dfn}[see {\cite[Proposition 5.6]{Asa-cover}}]
\label{dfn:omega}
Let $\mathcal{B}\in\GrCat$ and let $(P,\psi)\colon\mathcal{B}\#G\to(\mathcal{B}\#G)/G$
be the canonical functor. We define a 1-morphism $\om_{\mathcal{B}}\colon\mathcal{B}\to(\mathcal{B}\#G)/G$
in $\GrCat$ as follows.\textbf{}\\
\textbf{On objects.} For each $x\in\mathcal{B}$ we set\[
\om_{\mathcal{B}}(x):=P(x^{(1)}).\]
\textbf{On morphisms.} For each $f\colon x\to y$ in $\mathcal{B}$,
we set\[
\om_{\mathcal{B}}(f):=P_{x^{(1)},y^{(1)}}^{(1)}(f).\]
Then we have already shown that $\om_{\mathcal{B}}$ is a strictly
degree-preserving equivalence of $G$-graded categories in \cite[Proposition 5.6]{Asa-cover}. \end{dfn}
\begin{lem}
$\om$ is a $2$-natural transformation.\end{lem}
\begin{proof}
Let $(H,r)\colon\mathcal{B}\to\mathcal{B}'$ be a 1-morphism in $\GrCat$
and $(P',\psi')\colon\mathcal{B}'\#G\to(\mathcal{B}'\#G)/G$ the canonical
functor. We define a natural transformation $\Phi_{(H,r)}$ in the
diagram
\[
\xymatrix@C=7em{
\mathcal{B}\ar[r]^{(H,r)} & \mathcal{B}'\\
(\calB\#G)/G\ar[r]_{((H,r)\#G)/G} & (\calB'\#G)/G\ar_{\om_{\calB}}"1,1";"2,1"\ar^{\om_{\calB'}}"1,2";"2,2"\ar@{=>}_{\Phi_{(H,r)}}^{\iso}"1,2";"2,1"}
\]
by
\[
\Phi_{(H,r)}x:=\phi'_{r_{x}}(Hx)^{(1)}
\]
for all $x\in\mathcal{B}$. Then it is not hard to verify that $\Phi_{(H,r)}$
is a natural isomorphism. This shows the 1-naturality of $\om$. Now
let $(H',r')\colon\mathcal{B}\to\mathcal{B}'$ be another 1-morphism
and $\th\colon(H,r)\To (H'r')$ a 2-morphism in $\GrCat$. Then it is easy
to check the commutativity of the diagram
\[
\xymatrix@C=5em{
\om_{\calB'}\cdot (H,r) & ((H,r)\#G)/G\cdot \om_{\calB}\\
\om_{\calB'}\cdot (H',r') & ((H',r')\#G)/G\cdot \om_{\calB}
\ar@{=>}^(0.4){\Phi_{(H,r)}}_(0.4){\iso}"1,1";"1,2"
\ar@{=>}_(0.4){\Phi_{(H',r')}}^(0.4){\iso}"2,1";"2,2"
\ar@{=>}^{\om_{\calB'}\cdot \th}"1,1";"2,1"
\ar@{=>}^{(\th\#G)/G\cdot\om_{\calB}}"1,2";"2,2"
}
\]
of natural transformations, which shows the 2-naturality of $\om$.
\end{proof}

\subsection{$\omega'\colon(?/G)(?\#G)\To\id_{G\text{-}GrCat}$}

\begin{dfn}[see Proposition \ref{prp:cov-fun-wrt-smash}]
\label{dfn:om'}
Let $\mathcal{B}\in\GrCat$ and let $(P,\psi)\colon\mathcal{B}\#G\to(\mathcal{B}\#G)/G$
be the canonical functor. We define a functor $\om'_{\mathcal{B}}\colon(\mathcal{B}\#G)/G\to\mathcal{B}$
as the unique functor that makes the diagram\[
\xymatrix{\calB\#G\ar[rr]^{Q}\ar[rd]_{(P,\psi)} &  & \calB\\
 & (\calB\#G)/G\ar@{-->}[ru]_{\om_{\mathcal{B}}'}}
\]
strictly commutative, where $Q$ is the canonical $G$-covering functor
associated to the smash product. Namely, $\om'_{\mathcal{B}}$ is
defined as follows.\textbf{}\\
\textbf{On objects}. For each $P(x^{(a)})\in(\mathcal{B}\#G)/G$
we set\[
\om'_{\mathcal{B}}(P(x^{(a)})):=x.\]
\textbf{On morphisms}. For each $P(x^{(a)}),P(y^{(b)})\in(\mathcal{B}\#G)/G$,
we have the following diagram:\[
\xymatrix{((\calB\#G)/G)(P(x^{(a)}),P(y^{(b)})) & \calB(x,y)\\
\Ds_{c\in G}(\calB\#G)(x^{(ca)},y^{(b)}) & \Ds_{c\in G}\calB^{b^{-1}ca}(x,y).\ar@{-->}"1,1";"1,2"\ar_{\wr}^{P_{x^{(a)},y^{(b)}}^{(1)}}"2,1";"1,1"\ar@{=}"2,1";"2,2"\ar@{=}"2,2";"1,2"}
\]
Using this we set\[
\om'_{\calB}(u):={P_{x^{(a)},y^{(b)}}^{(1)}}^{-1}(u)\]
 for all $u\in((\calB\#G)/G)(P(x^{(a)}),P(y^{(b)}))$.\\
\textbf{Degree adjuster}. Finally we define a degree adjuster $r_{\mathcal{B}}$
of $\om'_{\mathcal{B}}$ by\[
r_{\mathcal{B}}(P(x^{(a)})):=a\]
for all $P(x^{(a)})\in(\mathcal{B}\#G)/G$. 
\end{dfn}

\begin{lem}
$\om'_{\calB}=(\om'_{\calB},r_{\calB})$
is a degree-preserving functor, and hence a $1$-morphism
in $\GrCat$ for all $\calB \in\GrCat$.
\end{lem}

\begin{proof}
It is not hard to verify that $\om'_\calB$
turns out to be a functor.
We show that $\om'_{\calB}=(\om'_{\calB},r_{\mathcal{B}})$ is degree-preserving
(see Definition \ref{dfn-graded-cat}).
Let $P(x^{(a)}), P(y^{(b)}) \in (\calB \# G)/G$ and $c \in G$.
Then
$$
\begin{aligned}
\om'_\calB(((\calB\# G/G)^{r_\calB(y^{(b)})\cdot c}(P(x^{(a)}), P(y^{(b)}))
&=\om'(P_{x^{(a)}, y^{(b)}}^{(1)}((\calB\# G)(A_{bc}x^{(a)}, y^{(b)})))\\
&=(\calB\# G)(x^{(bca)}, y^{(b)})\\
&=\calB^{b\inv bca}(x,y)=\calB^{ca}(x,y)\\
&=\calB^{c\cdot r_\calB(x^{(a)})}(\om'_\calB(P(x^{(a)}), \om'_\calB(P(y^{(b)}))).
\end{aligned}
$$
\end{proof}

\begin{rmk}[cf.\ {\cite[Remark 5.9]{Asa-cover}}]
\label{rmk:necessity-weak}
As is seen above $\om'_\calB$ is not strictly degree-preserving in general.
This forced us to extend the definition of degree-preserving functors
from a strict version to a weak one.
\end{rmk}

\begin{lem}
$\om'$ is a strictly $2$-natural transformation.\end{lem}
\begin{proof}

Let $(H,r)\colon\mathcal{B}\to\mathcal{B}'$ be a 1-morphism in $\GrCat$
and $(P,\psi)\colon\mathcal{B}\#G\to(\mathcal{B}\#G)/G$,
$(P',\psi')\colon\mathcal{B}'\#G\to(\mathcal{B}'\#G)/G$ the canonical
functors.
We first show the 1-naturality of $\om'$, i.e., the commutativity of the diagram
$$
\xymatrix@C=10ex{
(\calB\# G)/G & (\calB'\# G)/G\\
\calB & \calB'.
\ar^{((H,r)\# G)/G}"1,1";"1,2"
\ar_{(H, r)}"2,1";"2,2"
\ar_{\om'_\calB}"1,1";"2,1"
\ar^{\om'_{\calB'}}"1,2";"2,2"
}
$$
To show this let $u\colon P(x^{(a)}) \to P(y^{(b)})$ be in $(\calB \# G)/G$
and $f:= P_{x^{(a)}, y^{(b)}}^{(1)}{}\inv (u)$.
Then
$$
[\om'_{\calB'}\circ ((H,r)\# G)/G](P(x^{(a)})) = \om'_{\calB'}(P'((Hx)^{(ar_x)}))
=Hx = [(H, r)\circ \om'_\calB](P(x^{(a)})),
$$
and
$$
\begin{aligned}
{[}\om'_{\calB'} \circ ((H,r)\# G)/G](u)
&\overset{(\mathrm{a})}{=} [((P', \ps')((H, r)\# G))^{(1)}_{x^{(a)}, y^{(b)}}](f)\\
&\overset{(\mathrm{b})}{=} (P', \ps')^{(1)}_{(Hx)^{(ar_x)}, (Hy)^{(br_y)}}(Hf)\\
& = Hf\\
&= [(H, r) \circ \om'_\calB](u),
\end{aligned}
$$
where  the equality (a) holds by definition of $((H,r)\# G)/G$
(see \eqref{eq:induced-funct} and Proof of \cite[Proposition 2.6 (3)]{Asa-cover}),
and the equality (b) follows from the fact that $(H, r)\# G$ is strictly
$G$-equivariant.

To show the 2-naturality of $\om'$ let $(H', r')\colon \calB \to \calB'$ be
another 1-morphism and $\th \colon (H, r) \To (H', r')$ a 2-morphism
 in $\GrCat$.
 It is enough to verify the following:
 $$
 \om'_{\calB'}((\th \# G)/G) = \th  \om'_\calB.
 $$
For each $P(x^{(a)}) \in (\calB \# G)/G$
we have
$$
\begin{aligned}
{[}\om'_{\calB'} ((\th \# G)/G)]P(x^{(a)}) &= \om'_{\calB'} ((\th \# G)/G)P(x^{(a)}))\\
&= \om'_{\calB'} (P'((\th \# G)(x^{(a)}))) \\
&= \om'_{\calB'}(P'(\th x))\\
&= \om'_{\calB'}(P'^{(1)}_{(Hx)^{(ar_x)},(Hy)^{(ar'_x)}}(\th x))\\
&= \th x\\
&= \th \om'_{\calB}(P(x^{(a)})).
\end{aligned}
$$
\end{proof}

\subsection{Remaining parts of the proof of Theorem \ref{main-thm}}

\subsubsection*{Verification of $(\ref{eq:ep'ep})$}

By definitions of $\ep$ and $\ep'$ the the equality (\ref{eq:ep'ep})
is obvious.

\subsubsection*{Verification of $(\ref{eq:epep'})$}

Let $\mathcal{C}\in\Cat$ and let $(P,\psi)\colon\mathcal{C}\to\mathcal{C}/G$
be the canonical functor. It is easy to see that we can define a natural
isomorphism $\Theta\colon\id_{(\mathcal{C}/G)\#G}\to\ep_{\mathcal{C}}\ep'_{\mathcal{C}}$
by\[
\Theta((Px)^{(a)}):=\psi_{a}x\]
for all $(Px)^{(a)}\in(\mathcal{C}/G)\#G$.

\subsubsection*{Verification of $(\ref{eq:om'om})$}

By definitions of $\om$ and $\om'$ the equality (\ref{eq:om'om})
is obvious.

\subsubsection*{Verification of $(\ref{eq:omom'})$}

Let $\mathcal{B}\in\GrCat$ and let $(P,\psi)\colon\mathcal{B}\#G\to(\mathcal{B}\#G)/G$
be the canonical functor. It is not hard to see that we can define
a natural isomorphism $\Xi\colon\om_{\mathcal{B}}\om'_{\mathcal{B}}\to\id_{(\mathcal{B}\#G)/G}$
by \[
\Xi(P(x^{(a)})):=\psi_{a}(x^{(1)})\]
for all $P(x^{(a)})\in(\mathcal{B}\#G)/G$.

The verifications that the pasting of (\ref{eq:sharp-G}) is equal
to the identity and that the pasting of (\ref{eq:slash-G}) is equal
to the identity are easy and are left to the reader.

This finishes the proof of Theorem \ref{main-thm}.$\qed$

\section{Equivalences in 2-categories $\Cat$ and $\GrCat$}

To distinguish several kinds of equivalences (resp.\ isomorphisms)
we call equivalences (resp.\ isomorphisms) between categories
{\em category equivalences}  (resp.\ {\em category isomorphisms}).
In this section we give characterizations of equivalences in the 2-categories
 $\Cat$ and $\GrCat$ and
examine relationships
\begin{enumerate}
\item[(a)] between $G$-equivariant functors that are category equivalences
and equivalences in the 2-category $\Cat$ (see Theorem \ref{eqvar-eq}), and
\item[(b)] between degree-preserving functors that are category equivalences
and equivalences in the 2-category $\GrCat$. (See Remark \ref{alternative-pf}(2).)
\end{enumerate}
Note that a category equivalence was characterized by 
a half of a pair of functors in mutually reverse directions,
namely a functor is a category equivalence if and only if it is a fully faithful, dense functor.
We give similar characterizations of equivalences in both 2-categories $\Cat$
and $\GrCat$.

\subsection{Equivalences in $\Cat$}

First we characterize equivalences in $\Cat$ in the following theorem.

\begin{thm}\label{eqvar-eq}
Let $(E,\rho)\colon\mathcal{C}\to\mathcal{C}'$ be a $G$-equivariant functor in $\Cat$.
Then the following are equivalent.
\begin{enumerate}
\item $(E, \rho)$ is an equivalence in $\Cat$;
\item $E$ is fully faithful and dense
$($i.e., $E$ is a category equivalence$)$.
\end{enumerate}
Thus what we called {\em $G$-equivariant equivalences} in earlier sections
are exactly the equivalences in $\Cat$.
\end{thm}

\begin{proof}
(1)\ \implies (2).
This is trivial.

(2)\ \implies (1).
Assume that $E$ is a category equivalence.
Then $E$ has a quasi-inverse $F \colon \calC' \to \calC$, which we may regard
as a right adjoint to $E$, and hence there exist a counit $\ep \colon EF \Rightarrow \id_{\calC}$
and a unit $\et \colon \id_{\calC'} \Rightarrow FE$, which are natural isomorphisms.
Since $(E, \ro)$ is $G$-equivariant, $\ro_a$ are natural isomorphisms for all $a \in G$.
Therefore we can construct $\la = (\la_a)_{a\in G}$ by the following commutative diagram:
$$
\xymatrix{
A_a F & FA_a\\
FEA_aF & FA_aEF.
\ar@{=>}^{\la_a}"1,1";"1,2"
\ar@{=>}_{F\ro_a\inv F}^{\iso}"2,1";"2,2"
\ar@{=>}_{\et A_aF}^{\iso}"1,1";"2,1"
\ar@{=>}_{FA_a\ep}^{\iso}"2,2";"1,2"
}
$$
By construction $\la_a$ are natural isomorphisms for all $a\in G$.
\begin{claim}
$(F, \la) \colon \calC' \to \calC$ is a $1$-morphism in $\Cat$.
\end{claim}
Indeed, let $a, b \in G$.
It is enough to show the commutativity of the diagram:
$$
\xymatrix{
A_bA_aF & A_bFA_a & FA_bA_a\\
A_{ba}F && FA_{ba}
\ar@{=>}^{A_b\la_a}"1,1";"1,2"
\ar@{=>}^{\la_bA_a}"1,2";"1,3"
\ar@{=>}_{\la_{ba}}"2,1";"2,3"
\ar@{=}"1,1";"2,1"
\ar@{=}"1,3";"2,3"
}
$$
This follows from the following commutative diagrams:
{\small
$$
\xymatrix{
&A_bA_aF & A_bFA_a & A_bFA_a &FA_bA_a & FA_{ba}\\
FEA_bA_aF & AbFEA_aF & A_bFA_aEF & FEA_bFA_a & FA_bEFA_a\\
& FEA_bFEA_aF &FEA_bEA_aEF\\
& FA_bEFEA_aF & FA_bEFA_aEF\\
& FAEFA_aEF & FA_bA_aEF &&& FA_bA_a\\
FA_bEA_aF && FA_bA_aEF
\ar@{=>}^{A_b\la_a}"1,2";"1,3"
\ar@{=}"1,3";"1,4"
\ar@{=>}^{\la_bA_a}"1,4";"1,5"
\ar@{=}"1,5";"1,6"
\ar@{=>}^{\sim\ro_a\inv\sim}"2,2";"2,3"
\ar@{=>}_{\sim\ro_b\inv\sim}"2,4";"2,5"
\ar@{=>}_{\sim\ro_a\inv\sim}"3,2";"3,3"
\ar@{=>}_{\sim\ep\sim}"5,2";"5,3"
\ar@{=>}_{\sim\ro_a\inv\sim}"6,1";"6,3"
\ar@{=>}_{\sim\ep}"5,3";"5,6"
\ar@{=>}_{\sim\ep}"3,3";"2,4"
\ar@{=>}_{\et\sim}"1,2";"2,1"
\ar@{=>}^{\sim\et\sim}"2,1";"3,2"
\ar@{=>}_{\sim\et\sim}"1,2";"2,2"
\ar@{=>}^{\sim\ep}"2,3";"1,3"
\ar@{=>}_{\et\sim}"1,4";"2,4"
\ar@{=>}_{\sim\ep\sim}"2,5";"1,5"
\ar@{=>}_{\et\sim}"2,2";"3,2"
\ar@{=>}_{\et\sim}"2,3";"3,3"
\ar@{=>}_{\sim\ro_b\inv\sim}"3,2";"4,2"
\ar@{=>}^{\sim\ro_b\inv\sim}"3,3";"4,3"
\ar@{=>}_{\sim\ro_a\inv\sim}"4,2";"5,2"
\ar@{=>}^{\sim\ep\sim}"4,3";"5,3"
\ar@{=}"1,6";"5,6"
\ar@{=>}_{\sim\ro_b\inv\sim}"2,1";"6,1"
\ar@{=}"5,3";"6,3"
\ar@{}|-{(*)}"4,1";"4,2"
}
$$
}
$$
\xymatrix{
FEA_bA_aF &FA_bEA_aF\\
FEA_{ba}F & FA_{ba}EF,
\ar@{=>}^{\sim\ro_b\inv\sim}"1,1";"1,2"
\ar@{=>}_{\sim\ro_{ba}\inv\sim}"2,1";"2,2"
\ar@{=}"1,1";"2,1"
\ar@{=>}^{\sim\ro_a\inv\sim}"1,2";"2,2"
}
$$
where the commutativity $(*)$ follows from the following commutative diagram:
$$
\xymatrix{
EA_bA_a & EA_bFEA_a\\
A_bEA_a & A_bEFEA_a\\
A_bEA_a\\
A_bA_aE & A_bEFA_aE.
\ar@{=>}^-{\sim\et\sim}"1,1";"1,2"
\ar@{=>}^-{\sim\et\sim}"2,1";"2,2"
\ar@{=>}^{\sim\ep\sim}"4,2";"4,1"
\ar@{=>}_{\ro_b\inv\sim}"1,1";"2,1"
\ar@{=>}^{\ro_b\inv\sim}"1,2";"2,2"
\ar@{=>}^{\sim\ro_a\inv}"2,2";"4,2"
\ar@{=}"2,1";"3,1"
\ar@{=>}_{\sim\ro_a\inv}"3,1";"4,1"
\ar@{=>}^{\sim\ep\sim}"2,2";"3,1"
}
$$
In the above the symbol $\sim$ stands for a functor that is uniquely determined in
the diagram.
\begin{claim}
$\ep \colon (E, \ro)(F, \la) \Rightarrow (\id_{\calC}, (\id_{A_a})_{a\in G})$ is a $2$-isomorphism
in $\Cat$.
\end{claim}
Indeed, it is enough to show that $\ep$ is a 2-morphism in $\Cat$, i.e.,
the following is commutative:
$$
\xymatrix{
A_aEF & A_a\id_{\calC'}\\
EFA_a & \id_{\calC'}A_a.
\ar@{=>}^{A_a\ep}_{\iso}"1,1";"1,2"
\ar@{=>}_{\ep A_a}^{\iso}"2,1";"2,2"
\ar@{=>}_{E\la_a \circ \ro_aF}"1,1";"2,1"
\ar@{=}"1,2";"2,2"
}
$$
This follows from the following commutative diagram:
$$
\xymatrix{
A_aEF\\
EA_aF & EA_aF & A_aEF\\
EFEA_aF & EFA_aEF & EFA_a & A_a.
\ar@{=>}_{\ro_a F}"1,1";"2,1"
\ar@{=>}_{\sim\et\sim}"2,1";"3,1"
\ar@{=}"2,1";"2,2"
\ar@{=>}_{\ro_a\inv F}"2,2";"2,3"
\ar@{=>}_{\sim\ro_a\inv \sim}"3,1";"3,2"
\ar@{=>}_{\sim\ep}"3,2";"3,3"
\ar@{=>}_{\ep\sim}"3,3";"3,4"
\ar@{=}"1,1";"2,3"
\ar@{=>}^{\ep\sim}"3,1";"2,2"
\ar@{=>}_{\ep\sim}"3,2";"2,3"
\ar@{=>}^{\sim\ep}"2,3";"3,4"
}
$$
\begin{claim}
$\et \colon (\id_{\calC'}, (\id_{A'_a})_{a\in G}) \Rightarrow (F, \la)(E, \ro)$ is a $2$-isomorphism
in $\Cat$.
\end{claim}
Indeed, it is enough to show that $\et$ is a 2-morphism in $\Cat$, i.e.,
the following is commutative:
$$
\xymatrix{
A_aFE & A_a\id_{\calC}\\
FEA_a & \id_{\calC}A_a.
\ar@{=>}_-{A_a\et}^{\iso}"1,2";"1,1"
\ar@{=>}^{\et A_a}_{\iso}"2,2";"2,1"
\ar@{=>}_{F\ro_a \circ \la_aE}"1,1";"2,1"
\ar@{=}"1,2";"2,2"
}
$$
This follow from the following commutative diagram:
$$
\xymatrix{
A_a\\
A_a FE & FEA_a\\
FEA_aFE & FA_aE \\
FA_aEFE &FA_aE & FEA_a.
\ar@{=>}_{A_a\et}"1,1";"2,1"
\ar@{=>}_{\et\sim}"2,1";"3,1"
\ar@{=>}_{F\ro_a\inv}"2,2";"3,2"
\ar@{=}"3,2";"4,2"
\ar@{=>}_{\sim\ro_a\inv\sim}"3,1";"4,1"
\ar@{=>}"1,1";"2,1"
\ar@{=>}_{\sim\ep\sim}"4,1";"4,2"
\ar@{=>}_{F\ro_a}"4,2";"4,3"
\ar@{=>}^{\et\sim}"1,1";"2,2"
\ar@{=}"2,2";"4,3"
\ar@{=>}_{\sim\et}"2,2";"3,1"
\ar@{=>}^{\sim\et}"3,2";"4,1"
}
$$
These three claims show that $(E, \ro)$ is an equivalence in $\Cat$.
\end{proof}

\begin{rmk}
It is now trivial that the $G$-equivariant equivalence $\ep_{\calC} \colon \calC \to (\calC/G)\# G$
 is an equivalence in $\Cat$ by the theorem above.
\end{rmk}

\subsection{Equivalences in $\GrCat$}
Next we characterize equivalences in $\GrCat$.
We first define necessary terminologies.

\begin{dfn}
Let $\calA$ be a category and $\calB$ a $G$-graded category.
\begin{enumerate}
\item Let $E, F \colon \calA \to \calB$ be functors.
Then a natural transformation $\ep \colon E \Rightarrow F$
is called {\em homogeneous} if $\ep_x \colon Ex \to Fx$
are homogeneous in $\calB$ for all $x \in \calA_0$.
\item Let $\calS$ be a subclass of $\calB_0$ and $\calB'$ a full subcategory of $\calB$
with $\calB'_0 = \calS$.
Then $\calS$ (or $\calB'$) is said to be {\em homogeneously dense} in $\calB$ if
for each $x \in \calB_0$ there exists an $x' \in \calS$ such that
there exists a homogeneous isomorphism $x \to x'$.
\item A functor $F \colon \calA \to \calB$
is said to be {\em homogeneously dense}
if the object class $F(\calA_0)$ is homogeneously dense in $\calB$.
\end{enumerate}
\end{dfn}

We give two examples of homogeneously dense subcategories,
the latter will be used to give an alternative proof of the fact that
$\om_{\calB} \colon \calB \to (\calB \# G)/G$ is an
equivalence in $\GrCat$ in Remark \ref{alternative-pf}(1).

Recall that a $\k$-algebra $A$ is called {\em local} if the sum of non-invertible elements
is non-invertible and that if $A$ is local, then 0 and 1 are its only idempotents.

\begin{exm}\label{exm:hmg-dense}
Let $\calB$ be a $G$-graded $\k$-category and
$(P,\psi)\colon\mathcal{B}\#G\to(\mathcal{B}\#G)/G$
be the canonical functor.
\begin{enumerate}
\item If $\calB(x, x)$ are local $\k$-algebras for all $x \in \calB_0$,
then any dense full subcategory $\calB'$ of $\calB$ is homogeneously dense.
\item Let $\calB'$ be the full subcategory
of $(\calB \# G)/G$ with $\calB'_0:= \om_{\calB}(\calB_0) =\{P(x^{(1)}) \mid x \in \calB\}$
(see Definition \ref{dfn:omega}).
Then $\calB'$ is homogeneously dense in $(\calB \# G)/G$.
Hence $\om_{\calB} \colon \calB \to (\calB \# G)/G$ is homogeneously dense.
\end{enumerate}
Indeed, to show the statement (1) it is enough to show that if $x \iso y$ in $\calB$, then there exists
a homogeneous isomorphism in $\calB(x, y)$.
Now let $f \colon x \to y$ be an isomorphism in $\calB$.
We may assume that $x \ne 0$.
Write $f$ and $f\inv$
as finite sums: $f = \sum_{a\in G}f_a$
and $f\inv = \sum_{b\in G}g_b$ with $f_a\in \calB^a(x, y)$ and $g_b \in \calB^b(y, x)$
for all $a, b \in G$.
Then $\sum_{a, b \in G}g_bf_a = \id_x$ shows that $h:=g_bf_a$ is an automorphism of $x$
for some $a, b \in G$ because $\calB(x, x)$ is a local algebra.
Thus $(h\inv g_b)f_a=\id_x$ and $e:=f_a(h\inv g_b)$ is an idempotent in
$\calB(x, x)$, and hence $e=\id_x$ or $e=0$.
But $(h\inv g_b)ef_a=\id_x \ne 0$ shows that $e \ne 0$.
Hence $f_a \colon x \to y$
is a homogeneous isomorphism.

The statement (2) follows from the fact that
$\ps_{a,\, x} \colon P(x^{(1)}) \to P(x^{(a)})$
are homogeneous isomorphisms of degree $a$ in $(\calB \# G)/G$
for all $x \in \calB_0$ and all $a \in G$
(see proof of \cite[p.\ 131, Claim 4]{Asa-cover} for $\deg \ps_{a,x}$).
\end{exm}

We now give a characterization of equivalences in the 2-category $\GrCat$.

\begin{thm}\label{eq-GrCat}
Let $(H,r) \colon\mathcal{B}\to\mathcal{A}$ be a degree-preserving functor in $\GrCat$.
Then the following are equivalent.
\begin{enumerate}
\item $(H,r)$ is an equivalence in $\GrCat$.
\item $H \colon \calB \to \calA$ is a category equivalence with
a quasi-inverse $I$ as a left adjoint both of whose counit $\ep \colon IH \Rightarrow \id_{\calA}$
and unit $\et \colon \id_{\calB} \Rightarrow HI$ are homogeneous
natural isomorphisms.
\item $H$ is fully faithful and homogeneously dense.
\end{enumerate}
In $(2)$, $I$ is made into a quasi-inverse $(I, s)$ of $(H, r)$ with $\ep$ the counit and
$\et$ the unit in a unique way.
The degree adjuster $s$ is given by
\begin{equation}\label{dfn-deg-adj}
s = (s_x)_{x \in \calA_0} \text{ with } s_x:= (\deg\et_x)\inv\,r_{Ix}\inv \in G
\text{ for all }x \in \calA_0.
\end{equation}
\end{thm}

\begin{proof}
(2)\ \implies (1).
Assume the statement (2).
Set $t_x:= \deg \ep_x$ for all $x \in \calB_0$, and $t'_x:= \deg \et_x$ for all $x \in \calA_0$.
Define $s$ as in \eqref{dfn-deg-adj}, i.e.,
$s_x:= {t'_x}\inv\,r_{Ix}\inv \in G$
for all $x \in \calA_0$.
\setcounter{claim}{0}
\begin{claim}
$(I, s) \colon \calA \to \calB$ is a 1-morphism in $\GrCat$.
\end{claim}
Indeed,
let $x, y \in \calA_0$, $a \in G$ and $f \in \calA^a(x, y)$.
It is enough to show that $If \in \calB^{s_y\inv a s_x}(Ix, Iy)$.
Since $\et$ is a natural transformation we have $HIf = \et_y f \et_x\inv \in \calA^{t'_y}(y, HIy)
\calA^a(x, y) \calA^{{t'_x}\inv}(HIx, x) \subseteq \calA^{t'_y a {t'_x}\inv(HIx, HIy)}$.
Since $H$ is fully faithful, $H$ induces a bijection $\calB(Ix, Iy) \to \calA(HIx, HIy)$,
which also induces bijections
$$\calB^b(Ix, Iy) \to \calA^{{r_{Iy}}\inv b\, r_{Ix}}(HIx, HIy)$$
for all $b \in G$.  Applying this to $b$ with ${r_{Iy}}\inv b r_{Ix} = t'_y a {t'_x}\inv$, we have
$$
If \in \calB^{r_{Iy} t'_y a {t'_x}\inv r_{Ix}\inv}(Ix, Iy) = \calB^{s_y\inv a s_x}(Ix, Iy).
$$
\begin{claim}
$\ep \colon (I, s)(H, r) \Rightarrow (\id_{\calB}, 1)$ is a $2$-isomorphism in $\GrCat$.
\end{claim}
Indeed,
it is enough to show that $\ep$ is a 2-morphism in $\GrCat$.
This is equivalent to saying that $t_x = r_x s_{Hx}$ for all $x \in \calB_0$
because $(I, s)(H, r) = (IH, (r_x s_{Hx})_{x \in \calB_0})$.
Let $x \in \calB_0$.  Then since $(H\ep x) (\et Hx) = \id_{Hx}$,
we have $1 = \deg(H\ep x)\deg(\et Hx) = r_x\inv t_x r_{IHx} t'_{Hx}$.
Hence $r_x s_{Hx} = r_x {t'_{Hx}}\inv r_{IHx}\inv =r_x r_x\inv t_x = t_x$, as desired.
\begin{claim}
$\et \colon (\id_{\calA}, 1) \To (H, r)(I, s)$ is a $2$-isomorphism in $\GrCat$.
\end{claim}
Indeed,
it is enough to show that $\et$ is a 2-morphism in $\GrCat$.
This is equivalent to saying that $t'_x = r_{Ix}\inv s_x\inv$ for all $x \in \calA_0$
because $(H, r)(I, s) = (HI, (s_x r_{Ix})_{x\in \calA_0})$.
By definition $r_{Ix}\inv s_x\inv = r_{Ix}\inv r_{Ix}t'_x = t'_x$, as desired.

These three claims show that $(H, r)$ is an equivalence in $\GrCat$.
By looking at the proof of Claim 3, we see that the degree adjuster $s$ of $I$
is uniquely determined as in \eqref{dfn-deg-adj} by $\et$ and $r$.

(1)\ \implies (3).
Assume the statement (1), and let $(I, s)$ be a quasi-inverse of $(H, r)$
with 2-isomorphisms $\ep \colon (I, s)(H, r) \To \id_{\calB}$
and $\et \colon \id_{\calA} \To (H, r)(I, s)$.
Then $H(\calB_0) \supseteq HI(\calA_0)$, and the latter is homogeneously dense
in $\calA$ because $\et \colon \id_{\calA} \To HI$ is a homogeneous natural
isomorphism.

(3)\ \implies (2).
Assume the statement (3).
We can imitate the proof of (iii)\ \implies (ii) in \cite[p.\ 93, Theorem 1]{Ma}
to construct a quasi-inverse $I \colon \calA \to \calB$ as a left adjoint to $H$ and
a pair of a counit $\ep \colon IH \To \id_{\calB}$ and a unit
$\et \colon \id_{\calA} \To HI$.
Here we just give definitions of them.
Then it is enough to show that both $\ep$ and $\et$ are homogeneous natural isomorphisms.

{\bf Definition of $I$ and $\et$.}
Let $x \in \calA_0$.  Since $H$ is homogeneously dense, there exists a $y_x \in \calB_0$
such that there is a homogeneous isomorphism $\et_x \colon x \to Hy_x$.
Choose a pair $(y_x, \et_x)$ once for all $x$, and define
$Ix:= y_x$.
Then $\et_x \colon x \to HIx$ is a homogeneous isomorphism, and define
$\et:= (\et_x)_{x \in \calA_0}$.

Let $f \in \calA(x, x')$.  we define $If$ as follows.
Since $H$ is fully faithful, $H$ induces a bijection
$
H_{Ix, Ix'} \colon \calB(Ix, Ix') \to \calA(HIx, HIx').
$
Then define $If:= H_{Ix, Ix'}\inv(\et_{x'}f\et_x\inv)$ as in the following diagram:
$$
\xymatrix{
x && x'\\
HIx && HIx'\\
Ix && Ix'.
\ar^f"1,1"; "1,3"
\ar^{\et_{x'}f\et_x\inv}"2,1"; "2,3"
\ar_{If}"3,1"; "3,3"
\ar_{\et_x}"1,1"; "2,1"
\ar^{\et_{x'}}"1,3"; "2,3"
\ar@{|->}^{H}"3,1"; "2,1"
\ar@{|->}_{H}"3,3"; "2,3"
\ar@{|->}_{H}"3,2"; "2,2"
\ar@{}|{\circlearrowright}"1,2";"2,2"
}
$$

{\bf Definition of $\ep$.}
 Let $y \in \calB_0$.
 Then $Hy \in \calA_0$, and $\et_{Hy} \in \calA(Hy, HIHy)$.
 $H$ induces a bijection $H_{IHy,\, y} \colon \calB(IHy, y) \to \calA(HIHy, Hy)$.
 Then define $\ep_y:= H_{IHy,\, y}\inv(\et_{Hy}\inv)$.
 
 Then the same proof as in \cite[p.\ 93]{Ma} works (or it is straightforward) to show that
 $I$ is a left adjoint functor to $H$ with the unit $\et \colon \id_{\calA} \to HI$ and the counit
 $\ep:= (\ep_y)_{y \in \calB_0} \colon IH \To \id_{\calB}$.
 
 Now by definition $\et$ is a homogeneous natural isomorphism.
 It remains to show that $\ep_y$ are homogeneous isomorphisms for all $y \in \calB_0$.
 Since $\et_{Hy}$ is a homogeneous isomorphism, so is $\et_{Hy}\inv \in \calA(HIHy, Hy)$.
 Set $a:=\deg \et_{Hy}$.
 Since $(H, r)$ is a degree-preserving functor, the bijection
 $H_{IHy,\, y}$ induces a bijection
 $$
 \calB^{r_y a\, r_{IHy}\inv}(IHy, y) \to \calA^{a}(HIHy, Hy) \ni \et_{Hy}\inv.
 $$
 Hence $\ep_y = H_{IHy,\, y}\inv(\et_{Hy}\inv) \in \calB^{r_y b\, r_{IHy}\inv}(IHy, y)$
 and is a homogeneous isomorphism.

\end{proof}

The following is  immediate by Theorem \ref{eq-GrCat}.

\begin{cor}\label{iso-GrCat}
Let $(H,r) \colon\mathcal{B}\to\mathcal{A}$ be a $1$-morphism in $\GrCat$.
Then the following are equivalent.
\begin{enumerate}
\item $(H,r)$ is an isomorphism in $\GrCat$.
\item $H \colon \calB \to \calA$ is a category isomorphism.
\end{enumerate}
If this is the case, then the inverse of $(H, r)$ is given by
$$
(H, r)\inv = (H\inv, (r_{H\inv x}\inv)_{x \in \calA_0}).
$$
\end{cor}

\begin{rmk}\label{alternative-pf}
Let $\calB \in \GrCat_0$.
\begin{enumerate}
\item Theorem \ref{eq-GrCat} and Example \ref{exm:hmg-dense}(2)
give an immediate alternative proof of the fact
that $\om_{\calB} \colon \calB \to (\calB \# G)/G$ is an equivalence in $\GrCat$.
\item Also by Theorem \ref{eq-GrCat}, a degree-preserving functor
$(H, r) \colon \calA \to \calB$ in $\GrCat$  with $H$ a category equivalence
is an equivalence in $\GrCat$
if and only if $H$ is homogeneously dense.
In particular,
if $\calB(x, x)$ are local algebras for all $x \in \calB_0$, then 
all degree-preserving functors that are category equivalences are equivalences in $\GrCat$ 
by Example \ref{exm:hmg-dense}(1).
\end{enumerate}
\end{rmk}

Next we will give one more characterization of an equivalence in $\GrCat$
using the composite of degree preserving functors which are surjective, bijective and injective
on objects.
First we add necessary terminologies.

\begin{dfn}
Let $\calB$ be a $G$-graded category.
\begin{enumerate}
\item For each $x, y \in \calB$, we say that $x$ and $y$ are {\em homogeneously isomorphic}
(and write $x \iso_H y$) if there exists a homogeneous isomorphism $x \to y$.
Since the set of homogeneous isomorphisms in $\calB$ is closed under composition
and taking inverses, the relation $\iso_H$ on $\calB_0$ is an equivalence relation,
whose equivalence classes are called {\em homogeneous isoclasses}.
\item 
Let $\calB'$ be a full subcategory of $\calB$.
Then $\calB'$ is called a {\em homogeneous skeleton} of $\calB$
if $\calB'_0$ forms a complete set of representatives of homogeneous isoclasses in $\calB_0$.
Note that $\calB'$ is homogeneously dense in $\calB$ if and only if
it contains a homogeneous skeleton of $\calB$.
\end{enumerate}
\end{dfn}

\begin{lem}\label{lem:hmg-dense}
Let $\calB \in \GrCat_0$.
If $\calB'$ is a homogeneously dense full subcategory of $\calB$,
then the inclusion functor $S \colon \calB' \incl \calB$ induces an equivalence 
$(S, 1) \colon \calB' \to \calB$ in $\GrCat$.
\end{lem}
 
\begin{proof}
Note that $\calB'$ is again a $G$-graded category by setting
$\calB'^a(x, y):= \calB^a(x, y)$ for all $x, y \in \calB'_0$ and all $a \in G$, and hence
$(S, 1) \colon \calB' \to \calB$ is a degree-preserving functor.
Then the assertion following by Theorem \ref{eq-GrCat}.
\end{proof}

\begin{prp}\label{eq-GrCat-decomp}
Let $(H, r) \colon \calB \to \calA$ be a degree-preserving functor in $\GrCat$.
Then the following are equivalent.
\begin{enumerate}
\item $(H, r)$ is an equivalence in $\GrCat$.
\item There exist homogeneously dense full subcategories $\calB'$ and $\calA'$
of $\calB$ and $\calA$, respectively and a homogeneous natural isomorphism
$$\ze \colon (H, r) \To (S', 1)(H', r')(N, s),$$
where $S \colon \calB' \incl \calB$ and $S' \colon \calA' \incl \calA$ are inclusion functors,
and $(N, s)$ is a quasi-inverse of the equivalence $(S, 1)$ in $\GrCat$.
\end{enumerate}
\end{prp}

\begin{proof}
(2)\ \implies (1).
This immediately follows by Theorem \ref{eq-GrCat}.

(1)\ \implies (2).
Assume the statement (1).
Then there exist degree-preserving functors
$(H, r) \colon \calB \to \calA$ and
$(I, s) \colon \calA \to \calB$, and 2-isomorphisms
$\ep \colon (I,s)(H, r) \To (\id_{\calB}, 1)$ and
$\et \colon (\id_{\calB}, 1) \To (H, r)(I,s)$ in $\GrCat$.
Let $\calB'$ be a {\em homogeneous skeleton} of $\calB$.
Then $\calB'$ is homogeneously dense in $\calB$.
Let $\calA'$ be the full subcategory of $\calA$ with $\calA'_0:= H(\calB'_0)$.
Then we claim that $\calA'$ is a homogeneous skeleton of $\calA$.
Indeed, let $x \in \calA_0$.
Then by construction there exist an $x' \in \calB'$ and
a homogeneous isomorphism $f \colon Ix \to x'$ in $\calB$.
Hence we have homogeneous isomorphisms
$x \ya{\et_x} HIx \ya{Hf} Hx'$ in $\calA$.
Thus $x \iso_H Hx' \in \calA'_0$,
which shows that $\calA'$ is
homogeneously dense in $\calA$.
Next assume that there exists a homogeneous isomorphism $g \colon Hx \to Hy$
for some $x, y \in \calB'_0$.
Then we have homogeneous isomorphisms
$x \xleftarrow{\ep_x} IHx \ya{Ig} IHy \ya{\ep_y} y$.
Thus $x \iso_H y$, and hence $x = y$.
As a consequence
\begin{equation}\label{injective}
 x \ne y \text{ implies } Hx \not\iso_H Hy.
\end{equation}
This proves the claim.
Now let
$S \colon \calB' \incl \calB$ and $S' \colon \calA' \incl \calA$ be inclusion functors,
and as in the proof of Theorem \ref{eq-GrCat}
construct a quasi-inverse $(N, s)$ of the equivalence $(S, 1)$ in $\GrCat$ as a left adjoint
with a counit $\id_{\id_{\calB'}} \colon NS = \id_{\calB'}$ and a unit $\nu \colon \id_{\calB} \To SN$.
Then $s_x=(\deg \nu_x)\inv$ for all $x \in \calB_0$.
The implication \eqref{injective} also shows that $H$ induces a bijection $\calB'_0 \to \calA'_0$.
As $H$ is fully faithful, $H$ induces a category isomorphism $H' \colon \calB' \to \calA'$
that satisfies $S'H' = HS$.  Let $r'$ be the restriction of $r$ to $\calB'_0$.
Then $(H', r') \colon \calB' \to \calA'$ is a degree-preserving functor,
which turns out to be an isomorphism in $\GrCat$ by Corollary \ref{iso-GrCat}.
Now $\ze:=H\nu$ is a homogeneous natural isomorphism $H \To HSN = S'H'N$.
It remains to show that $\ze$ is a 2-morphism
$(H, r) \To (S', 1)(H', r')(N, s) = (S'H'N, (s_xr'_{Nx})_{x\in \calB_0})$
in $\GrCat$.
For this it is enough to show that $\deg H\nu_x = (s_xr'_{Nx})\inv r_x$
for all $x \in \calB_0$.
Now since $\deg \nu_x = s_x\inv$ and $\nu_x \colon x \to SNx = Nx$, 
we have $\deg H\nu_x = r_{Nx}\inv s_x\inv r_x = (s_xr'_{Nx})\inv r_x$,
as desired.
\end{proof}

The following is immediate by Proposition \ref{eq-GrCat-decomp} and Lemma \ref{lem:hmg-dense}.

\begin{cor}
Let $\calB, \calA \in \GrCat_0$.
Then the following are equivalent.
\begin{enumerate}
\item $\calB \simeq \calA$ in $\GrCat$.
\item There exist homogeneously dense full subcategories $\calB'$ and $\calA'$
of $\calB$ and $\calA$, respectively such that $\calB' \iso \calA'$ in $\GrCat$.
\end{enumerate}

\end{cor}

\end{document}